\documentclass{amsart}
\usepackage{amsmath}
\usepackage{dsfont}
  \usepackage{paralist}
  \usepackage{graphics} 
  \usepackage{epsfig} 
\usepackage{graphicx}  \usepackage{epstopdf}
 \usepackage[colorlinks=true]{hyperref}
 \usepackage{amsmath}
\usepackage{amsfonts}
\usepackage{amssymb}
\usepackage{ mathrsfs }
\usepackage{amsthm}
\usepackage[pagewise]{lineno}
\usepackage{enumitem}
\usepackage{graphicx,color,xcolor,tikz}
\usepackage{bm}
\usepackage{hyperref}
\hypersetup{
    colorlinks=true,                         
    linkcolor=blue, 
    citecolor=red, 
  } 
\hypersetup{urlcolor=blue, citecolor=red}

\newtheorem{thm}{Theorem}
\newtheorem{prop}[thm]{Proposition}
\newtheorem{lem}[thm]{Lemma}

\newtheorem{defi}[thm]{Definition}
\newtheorem{rem}[thm]{Remark}

\def\A{\mathcal A}

\def\I{\mathcal I}
\def\M{\mathcal M}

\def\R{\mathbb R}
\def\T{\mathcal T}

\def\e{\varepsilon}

\newcommand{\pt}{\partial}

\newcommand{\abs}[1]{\ensuremath{\left|#1\right|}}
\newcommand{\card}[1]{\left\lvert#1\right\rvert}
\newcommand{\norm}[1]{\ensuremath{\left\|#1\right\|}}
\makeatletter
\newcommand{\doublewidetilde}[1]{{%
  \mathpalette\double@widetilde{#1}%
}}
\newcommand{\double@widetilde}[2]{%
  \sbox\z@{$\m@th#1\widetilde{#2}$}%
  \ht\z@=.9\ht\z@
  \widetilde{\box\z@}%
}
\makeatother

\title[Three scale unfolding homogenization method applied to cardiac bidomain model]{Three scale unfolding homogenization method applied to cardiac bidomain model}


\author{Fakhrielddine Bader$^*$ }
\address[Fakhrielddine Bader]{Mathematics Laboratory, Doctoral school of Sciences and Technologies, Lebanese University, Hadat Beirut, Lebanon \&  Laboratoire de Mathématiques Jean Leray, École Centrale de Nantes, Nantes, France}
\email{fakhrielddine.bader@ec-nantes.fr}

\author{Mostafa Bendahmane}
\address[Mostafa Bendahmane]{Institut de Mathématiques de Bordeaux and INRIA-Carmen Bordeaux Sud-Ouest, Université de Bordeaux, 33076 Bordeaux Cedex, France}
\email{mostafa.bendahmane@u-bordeaux.fr}

\author{Mazen Saad}
\address[Mazen Saad]{Laboratoire de Mathématiques Jean Leray, UMR 6629 CNRS, École Centrale de Nantes,  1 rue de Noé, 44321 Nantes, France}
\email{mazen.saad@ec-nantes.fr}

\author{Raafat Talhouk}
\address[Raafat Talhouk]{Mathematics Laboratory, Doctoral school of Sciences and Technologies, Lebanese University, Hadat Beirut, Lebanon}
\email{rtalhouk@ul.edu.lb}

\thanks{$^*$ Corresponding author: fakhri.bader.fb@gmail.com}

\begin{document}
\maketitle



\begin{abstract}
In this paper, we are dealing with a rigorous homogenization result at two different levels for the bidomain model of cardiac electro-physiology.
The first level associated with the mesoscopic structure such that the cardiac tissue consists of extracellular and intracellular domains separated by an interface (the sarcolemma). The second  one related to the microscopic structure in such a way that the intracellular medium can only be viewed as a periodical layout of unit cells (mitochondria). At the interface between intra- and extracellular media, the fluxes are given by nonlinear functions of ionic and applied currents. A rigorous homogenization process based on unfolding operators is applied to derive the macroscopic (homogenized) model from our meso-microscopic bidomain model. We apply a three-scale unfolding method in the intracellular problem to obtain its homogenized equation at two levels. The first level upscaling of the intracellular structure yields the mesoscopic equation. The second step of the homogenization leads to obtain the intracellular homogenized equation. To prove the convergence of the nonlinear terms, especially those defined on the microscopic interface, we use the boundary unfolding method and a Kolmogorov-Riesz compactness's result. Next, we use the standard unfolding method to homogenize the extracellular problem. Finally, we obtain, at the limit, a reaction-diffusion system on a single domain (the superposition of the intracellular and extracellular media) which contains the homogenized equations depending on three scales. Such a model is widely used for describing the macroscopic behavior of the cardiac tissue, which is recognized to be an important messengers between the cytoplasm (intracellular) and the other extracellular inside the biological cells.
\end{abstract}
\tableofcontents
\section{Introduction}
The heart is a hollow muscle whose role is to pump blood to the body’s organ through blood vessels. It is located near the centre of the thoracic cavity between the right and left lungs. It is a muscular organ which can be viewed as two pumps that operate in series. The right atrium and ventricle, which pump blood from systemic circulation through the superior and inferior vena cava, to the lungs where it receives oxygen. While the left atrium and ventricle, which pumps blood from the pulmonary circulation through the aorta, to every other part of the body. These four cavities are surrounded by a cardiac tissue that is organized into muscle fibers (see Figure \ref{heart} and see \cite{katz10} for more information). Specifically, the cardiac muscle cells are connected together to form an electrical syncytium, with tight electrical and mechanical connections between adjacent cardiac muscle cells. These fibers form a network of cardiac muscle cells called "cardiomyocytes" connected end-to-end by the gap junctions. The mechanical contraction of the heart is caused by the electrical activation of these myocardial cells. For more details about the physiology of heart, we refer to \cite{berne09} and about the electrical activity of heart we refer to \cite{char}.

\begin{figure}[h!]
  \centering
  \includegraphics[width=7cm]{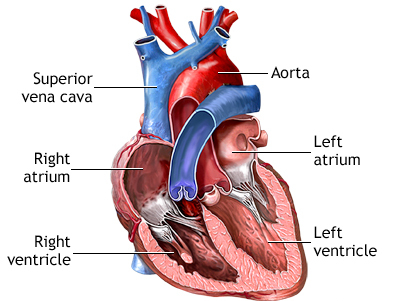}\quad \includegraphics[width=6cm]{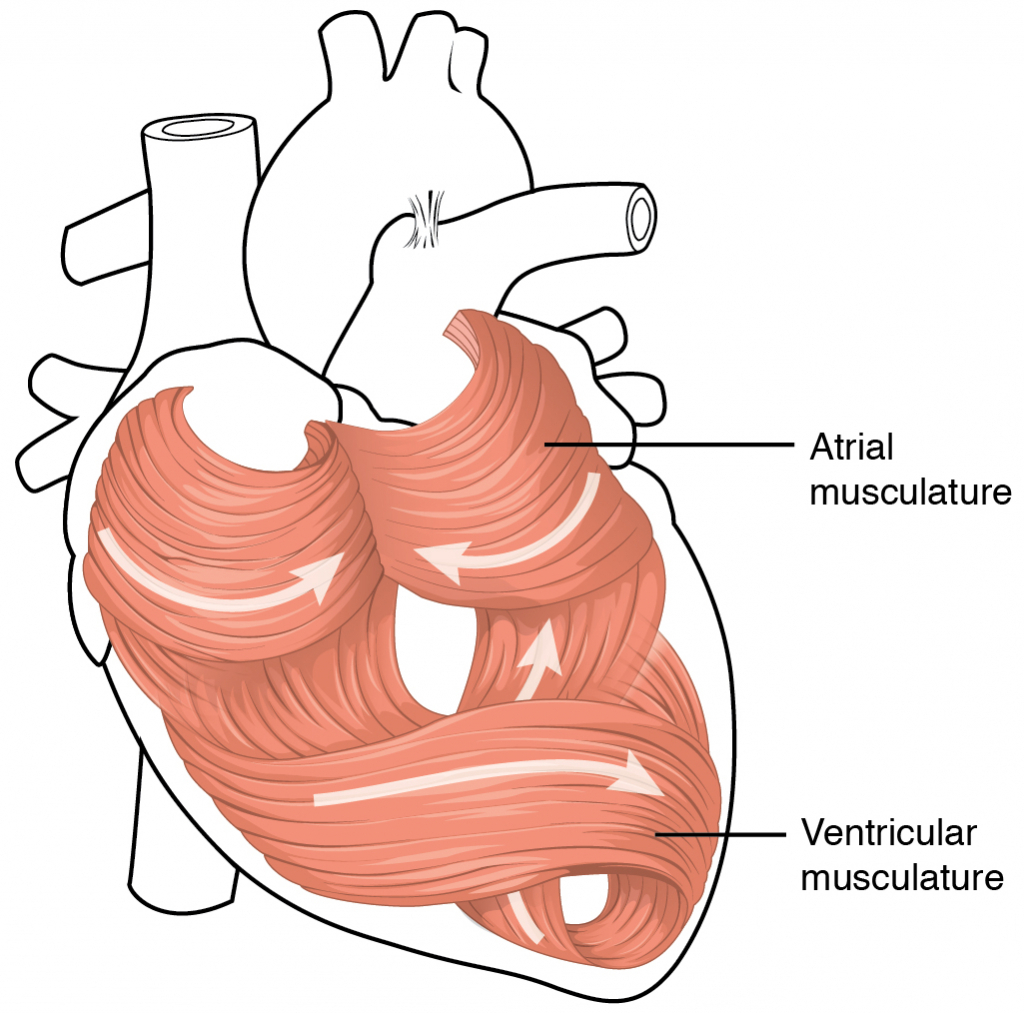}
  \caption{Composition of the heart}
  \url{http://pennstatehershey.adam.com/content.aspx?productid=114&pid=2&gid=19612, http://ressources.unisciel.fr/physiologie/co/4a_1.html }
  \label{heart}
 \end{figure}
 
  In this paper, our attention is initially directed at the organization of cardiac muscle cells within the heart. The structure of cardiac tissue studied in this paper is characterized at three different scales (see Figure \ref{cardio}). At mesoscopic scale, the cardiac tissue is divided into two media: one contains the contents of the cardiomyocytes, in particular the "cytoplasm" which is called the "intracellular" medium, and the other is called extracellular and consists of the fluid outside the cardiomyocytes cells. These two media are separated by a cellular membrane (the sarcolemma) allowing the penetration of proteins, some of which play a passive role and others play an active role powered by cellular metabolism. At microscopic scale, the cytoplasm comprises several organelles such as mitochondria. Mitochondria are often described as the "energy powerhouses" of cardiomyocytes and are surrounded by another membrane. In our study, we consider only that the intracellular medium can be viewed as a periodic structure composed of other connected cells. While at the macroscopic scale, this domain is well considered as a single domain (homogeneous).
 \begin{figure}[h!]
  \centering
  \includegraphics[width=13cm]{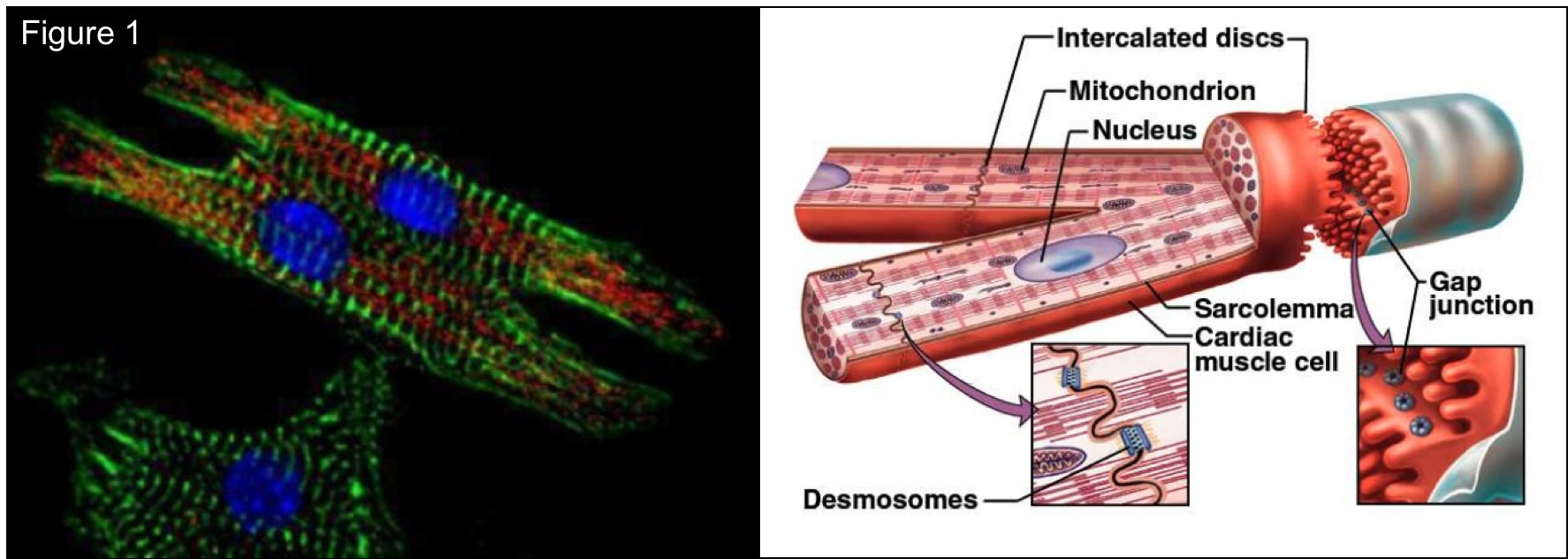}
  \caption{Representation of the cardiomyocyte structure}
  \url{http://www.cardio-research.com/cardiomyocytes}
  \label{cardio}
 \end{figure}
 It should be noted that there is a difference between the chemical composition of the cytoplasm and that of the extracellular medium. This difference plays a very important role in cardiac activity. In particular, the concentration of anions (negative ions) in cardiomyocytes is higher than in the external environment. This difference of concentrations creates a transmembrane potential, which is the difference in potential between these two media. The  model that describes the electrical activity of the heart, is called by "Bidomain model". The authors in \cite{bendunf19} established the well-posedness of this microscopic bidomain model under different conditions and proved the existence and uniqueness of their solutions (see also \cite{henri,marco,char09} for other approaches).\\ 
 We start from the microscopic bidomain model, resolving the geometry of the domain, which consists of two quasi-static approximation of elliptic equations, one for the electrical potential in the intracellular medium and one for the extracellular medium, coupled through a dynamical boundary equation at the interface of the two regions (the sarcolemma).  Our meso-microscopic model involves on two small scaling parameters $\e$ and $\delta$ whose are respectively the ratio of the microscopic and mesoscopic scales from the macroscopic scale \cite{henri}. 
 For more details about the modeling of the electrical activity of the heart, the reader is referred to \cite{colli,char}.
 
  Our goal in this paper is to rigorously derive, using the unfolding homogenization method, the macroscopic (homogenized) model of the cardiac tissue which is an approximation of the microscopic bidomain one and consists of equations formulated on the macroscopic scale. The macroscopic model consists of a system of reaction-diffusion equations with homogenized coefficients, approximating the microscopic solution on the two connected components of the domain. In general, the homogenization theory is the analysis of the macroscopic behavior of biological tissues by taking into account their complex microscopic structure. For an introduction to this theory, we cite \cite{sanchez}, \cite{doina},\cite{tartarintro} and \cite{bakhvalov}. Applications of this technique can also be found in modeling solids, fluids, solid-fluid interaction, porous media, composite materials, cells and cancer invasion. This technique also has an interest in the field of numerical analysis where various new computational techniques (finite difference, finite elements and finite volume methods) have been developed, we cite for instance \cite{abdfinite},\cite{bellofinite}. Several methods are related to this theory. Classically, homogenization has been done by the multiple-scale method which was first introduced by Benssousan et al. \cite{ben} and by Sanchez-Palencia \cite{sanchez} for linear and periodic operators. 
  The two-scale convergence method introduced by Nugesteng \cite{ngu} and developped by Allaire et al. \cite{allaire92}. In addition, Allaire and Briane \cite{allairebriane}, Trucu et al. \cite{trucu} introduced a further generalization of the previous method via a three-scale convergence approach for distinct problems. 
   Recently, the periodic unfolding method was introduced by Cioranescu et al. in \cite{doinaunf02} for the study of classical periodic homogenization in the case of fixed domains and adapted to homogenization in domains with holes by Cioranescu et al. \cite{doinaunf06,doinaunf12}. The unfolding reiterated homogenization method was studied first by Meunier and Van Schaftingen \cite{meunier05} for nonlinear partial differential equations with oscillating coefficients and multi scales. The unfolding method is essentially based on two operators: the first represents the unfolding operator and the second operator consists to separate the microscopic and macroscopic scales. The idea of the unfolding operator was introduced firstly in \cite{douglas} under the name "dilation" operator. The name "unfolding operator" was then introduced in \cite{doinaunf02} and deeply studied in \cite{doinaunf06,doinaunf08,doinaunf12}. The interest of this method comes from the fact that we use standard weak or strong convergences in $L^p$ spaces. On the other hand, the unfolding operator  maps functions defined on oscillating domains into functions defined on fixed domains. Hence, the proof of homogenization results be comes quite simple.
   

  Now, we mention some different homogenization methods that are applied to the microscopic bidomain model to obtain the macroscopic bidomain model. Krassowska and Neu \cite{neukra} applied the two-scale asymptotic method to formally obtain this macroscopic model (see also \cite{amar06,henri} for different approaches). Furthermore, Pennachio et al. \cite{colli05} used the tools of the $\Gamma$-convergence method to obtain a rigorous mathematical form of this homogenized macroscopic model. Amar et al. \cite{amar13} studied a hierarchy of electrical conduction problems in biological tissues via two-scale convergence. Recently, the authors in \cite{bendunf19} proved the existence and uniqueness of solution of the microscopic bidomain model based on Faedo-Galerkin method. Further, they used the unfolding homogenization method at two scales to show that the solution of the microscopic biodmain model converges to the solution of the macroscopic one.
  
  
  \textit{The main of contribution of the present paper.} The cardiac tissue structure viewed as double-periodic domain and studied at the three different (micro-meso-macro) scales . The aim is to derive the macroscopic (homogenized) bidomain model of cardiac electro-physiology from the microscopic bidomain model. This paper presents a rigorous mathematical justification for the results obtained in a recent work \cite{BaderDev} based on a three-scale asymptotic homogenization method. For this, we will apply a three-scale unfolding method on the intracellular problem by accounting two scaling parameters $\e$ and $\delta$ to obtain its homogenized equation. Further, to pass to the limit in nonlinear terms, we use the technique involving the unfolding operator introduced in \cite{mariahom} and a Kolmogorov-Riesz compactness's result. Second, we will follow the standard unfolding method on the extracellular one (similar derivation may be found in \cite{bendunf19}). An important motivation for our investigations is their application to unfolding homogenization method proposed to the effective properties of the cardiac tissue at micro-meso-macro scales.  
   
  
  \textit{The outline of the paper is as follows.} In Section \ref{geobid}, we give a precise description of the geometry of cardiac tissue and introduce the microscopic bidomain model featured by two parameters, $\e$ and $\delta,$ characterizing the  micro and mesoscopic scales. Furthermore, the existence of a unique weak solution for the microscopic problem is stated and a priori estimates for the microscopic solutions are derived. Moreover, the main result is presented in this section. In Section \ref{methodunf}, we recall the notion of the unfolding operator and the convergence results used for homogenizaton. Section \ref{unf} is devoted to homogenization procedure. The three-scale unfolding method applied in the intracellular problem is explained in Subsection \ref{unfintra}. The homogenized equation for the intracellular problem is obtained, at two levels, in terms of the coefficients of conductivity matrices and cell problems. The first level of homogenization yields the mesoscopic problem and then we complete the second level to obtain the corresponding homogenized equation. In Subsection \ref{unfextra}, the homogenized equation for the extracellular problem is obtained at one level using a standard unfolding method. Finally, in Subsection \ref{macro}, we derive the macroscopic bidomain model. 
Furthermore, in Appendix \ref{appB}, we give a compactness result for the space $L^{2}(\Omega, B)$ with a Banach space $B$ and $\Omega\subset \R^{d}$.

\section{Bidomain modeling of the heart tissue}\label{geobid}
In this section we define the geometry of cardiac tissue and we present our meso-microscopic bidomain model.
\subsection{Geometry of heart tissue}
The cardiac tissue $\Omega \subset \R^d$ is considered as a heterogeneous periodic domain with a Lipschitz boundary $\pt \Omega$. The structure of the tissue is periodic at meso- and microscopic scales related to two small parameters $\e$ and $\delta$, respectively, see Figure \ref{three_scale}. 

 \begin{figure}[h!]
  \centering
  \includegraphics[width=15cm]{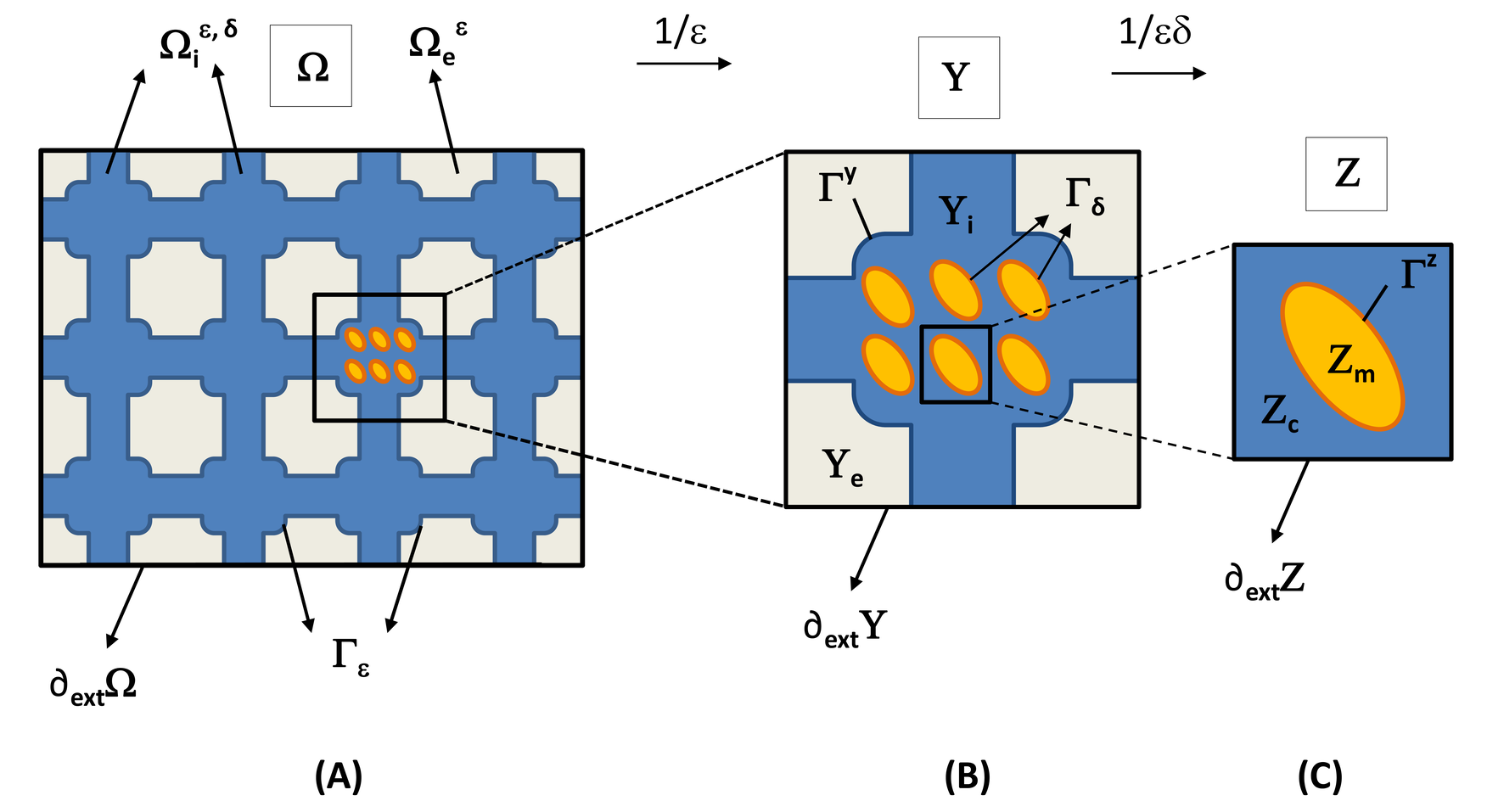}
  \caption{(A) Periodic heterogeneous domain $\Omega$, (B) Reference cell $Y$ at $\e$-structural level and (B) Reference cell $Z$ at $\delta$-structural level }
  \label{three_scale}
 \end{figure} 

 Following the standard approach of the homogenization theory, this structure is featured by two parameters $\ell^\text{mes}$ and $\ell^\text{mic}$ characterizing, respectively, the mesoscopic and microscopic length of a cell in meso- or microscopic domain. Under the two-level scaling, the characteristic lengths $\ell^\text{mes}$ and $\ell^\text{mic}$ are related to a given macroscopic length $L$ (of the cardiac fibers), such that the two scaling parameters $\e$ and $\delta$ are introduced by:
 $$\e=\frac{\ell^\text{mes}}{L} \text{ and } \delta=\frac{\ell^\text{mic}}{L} \ \text{ with } \ \ell^\text{mic}<<\ell^\text{mes}.$$

\paragraph{\textbf{The mesoscopic scale.}} The domain $\Omega$ is composed of two ohmic volumes, called intracellular $\Omega_i^{\e,\delta}$ and extracellular $\Omega_e^{\e}$ medium (see \cite{colli05} for two-scale approach). Geometrically, we find that $\Omega_i^{\e,\delta}$ and $\Omega_e^{\e}$ are two open connected regions such that:
 $$\overline{\Omega}=\overline{\Omega}_{i}^{\e,\delta}\cup \overline{\Omega}_e^{\e}, \ \text{with} \ \Omega_i^{\e,\delta}\cap\Omega_e^{\e}=\emptyset.$$ These two regions are separated by the surface membrane $\Gamma_{\e}$ which is expressed by: $$\Gamma_{\e}=\pt \Omega_i^{\e,\delta} \cap \pt \Omega_e^{\e},$$ assuming that the membrane is regular. We can observe that the domain $\Omega_i^{\e,\delta}$ as a perforated domain obtained from $\Omega$ by removing the holes which correspond to the extracellular domain $\Omega_e^{\e}.$ 
 
  At this $\bm{\varepsilon}$-structural level, we can divide $\Omega$ into $N_{\e}$ small elementary cells $Y_{\e}=\overset{d}{\underset{n=1}{\prod }}]0,\e \, \ell^\text{mes}_n[,$  with $\ell^\text{mes}_1,\dots,\ell^\text{mes}_d$ are positive numbers. These small cells are all equal, thanks to a translation and scaling by $\e,$ to the same cell of periodicity called the reference cell  $Y=\overset{d}{\underset{n=1}{\prod }}]0,\ell^\text{mes}_n[$. Next, we denote by $T_\e^k$ a translation of $\e k$   with $k=( k_1,\dots, k_d ) \in\mathbb{Z}^d$. Note that if the cell considered $Y^k_\e$ is located at the $k_n^{\text{\text{ième}}}$ position according to the direction $n$ of space considered, we can write:
\begin{equation*}
 Y^k_{\e}:=T^k_\e+\e Y=\lbrace \e \xi : \xi \in k_\ell+Y \rbrace,
 \end{equation*}
 with $ k_\ell:=( k_1\ell^\text{mes}_1,\dots,  k_d \ell^\text{mes}_d ).$\\
 Therefore, for each macroscopic variable $x$ that belongs to $\Omega,$ we define the corresponding mesoscopic variable $y\approx\dfrac{x}{\e}$ that belongs to $Y$ with a  translation. Indeed, we have:
 \begin{equation*}
 x \in \Omega \Rightarrow \exists k \in\mathbb{Z}^d  \ \text{ such that }  \ x \in Y^k_\e \Rightarrow x=\e (k_\ell+y) \Rightarrow y=\dfrac{x}{\e}-k_\ell \in Y.
 \end{equation*}

 Since, we will study in the extracellular medium $\Omega^{\e}_e$ the behavior of the functions $u(x,y)$ which are $\textbf{y}$-periodic, so by periodicity we have $u\left( x,\dfrac{x}{\e}-k_\ell\right)  =u\left( x,\dfrac{x}{\e}\right) .$ By notation, we say that $y=\dfrac{x}{\e}$ belongs to $Y.$

 We are assuming that the cells are periodically organized as a regular network of interconnected cylinders at the mesoscale. The mesoscopic unit cell $Y$ splits into two parts: intracellular $Y_i$ and extracellular $Y_e.$ These two parts are separated by a common boundary $\Gamma^{y}.$ So, we have:
 \begin{equation*}
 Y=Y_i \cup Y_e \cup \Gamma^{y}, \quad \Gamma^{y}= \pt Y_i \cap \pt Y_e.
 \end{equation*}  
In a similar way, we can write the corresponding common periodic boundary as follows:
  \begin{equation*}
 \Gamma^k_{\e}:=T^k_\e+\e \Gamma^{y}=\lbrace \e \xi : \xi \in k_\ell+\Gamma^{y} \rbrace,
 \end{equation*}
 with   $T^k_\e$ denote the same previous translation.

 In summary, the intracellular and extracellular medium at mesoscale can be described as the intersection of the cardiac tissue $\Omega$ with the cell $Y^k_{j,\e}$ for $j=i,e:$  
 \begin{equation*}
\Omega^{\e}_i=\Omega \cap \underset{k\in \mathbb{Z}^d}{\bigcup} Y^k_{i,\e}, \quad \Omega^{\e}_e=\Omega \cap \underset{k\in \mathbb{Z}^d}{\bigcup} Y^k_{e,\e}, \quad \Gamma_{\e}=\Omega \cap \underset{k\in \mathbb{Z}^d}{\bigcup} \Gamma^k_{\e},
\end{equation*}
with each cell defined by $Y^k_{j,\e}=T^k_\e+\e Y_j $ for $j=i,e$.

\paragraph{\textbf{The microscopic scale.}}The cytoplasm  contains far more mitochondria described as "the powerhouse of the myocardium" surrounded by another membrane $\Gamma_{\delta}.$ Then, we only assume that the intracellular medium $\Omega_i^{\e,\delta}$ can also be viewed as a periodic perforated domain.
 
 At this $\bm{\delta}$-structural level, we can divide this medium with the same strategy into small elementary cells $Z_{\delta}=\overset{d}{\underset{n=1}{\prod }}]0,\delta \, \ell^\text{mic}_n[,$ with $\ell^\text{mic}_1,\dots,\ell^\text{mic}_d$ are positive numbers. Using a similar translation (noted by $T^{k'}_\delta$), we return to the same reference cell noted by $Z=\overset{d}{\underset{n=1}{\prod }}]0,\ell^\text{mic}_n[.$ Note that if the cell considered $Z^{k'}_{\delta}$ is located at the $k_n^{'\text{ième}}$ position according to the direction $n$ of space considered, we can write:
\begin{equation*}
 Z^{k'}_{\delta}:=T^{k'}_\delta+\delta Z=\lbrace \delta \zeta : \zeta \in k'_{\ell'}+Z \rbrace,
 \end{equation*}
 with $ k'_{\ell'} :=( k'_1\ell^\text{mic}_1,\dots,  k'_d \ell^\text{mic}_d ).$ \\
 Therefore, for each macroscopic variable $x$ that belongs to $\Omega,$ we also define the corresponding microscopic variable $z\approx\dfrac{y}{\delta}\approx\dfrac{x}{\e\delta}$ that belongs to $Z$ with the translation $T^{k'}_\delta$. 
 

  The microscopic reference cell $Z$ splits into two parts: mitochondria part $Z_{m}$ and the complementary part $Z_{c}:=Z \setminus Z_{m}.$ These two parts are separated by a common boundary $\Gamma^{z}.$ So, we have:
 \begin{equation*}
 Z=Z_{m} \cup Z_{c} \cup \Gamma^{z}, \quad \Gamma^{z}= \pt Z_{m}.
 \end{equation*}
By definition, we have $\pt Z_{c}=\pt_{\text{ext}} Z \cup \Gamma^{z}.$
 
  More precisely, we can write the intracellular meso- and microscopic domain $\Omega_i^{\e,\delta}$ as follows:
\begin{equation*}
 \Omega^{\e,\delta}_i=\Omega \cap \underset{k\in \mathbb{Z}^d}{\bigcup} \left( Y^k_{i,\e} \cap \underset{k'\in \mathbb{Z}^d}{\bigcup} Z^{k'}_{c,\delta}\right) 
 \end{equation*}
with $Z^{k'}_{c,\delta}$ is defined by:
\begin{equation*}
 Z^{k'}_{c,\delta} :=T^{k'}_\delta+\delta Z_{c}=\lbrace \delta \zeta : \zeta \in k'_{\ell'}+Z_{c} \rbrace.
 \end{equation*}
  In the intracellular medium $\Omega^{\e,\delta}_i,$ we will study the behavior of the functions  $u(x,y,z)$ which are $\textbf{z}$-periodic, so by periodicity we have $u\left( x,y,\dfrac{x}{\e \delta}-\dfrac{k_{\ell}}{\delta}-k'_{\ell'}\right) =u\left( x,y,\dfrac{x}{\e\delta}\right).$ By notation, we say that $z=\dfrac{x}{\e\delta}$ belongs to $Z.$

 Similarly, we describe the common boundary at microscale as follows: 
 \begin{equation*}
\Gamma_{\delta}=\Omega \cap \underset{k'\in \mathbb{Z}^d}{\bigcup} \Gamma^{k'}_{\delta},
\end{equation*}
where $\Gamma^{k'}_{\delta}$ given by:
\begin{equation*}
 \Gamma^{k'}_{\delta}:=T^{k'}_\delta+\delta \Gamma^z=\lbrace \delta \zeta : \zeta \in k'_{\ell'}+\Gamma^z \rbrace,
 \end{equation*}
 with   $T^{k'}_\delta$ denote the same previous translation.

\subsection{Microscopic bidomain model}
A vast literature exists on the bidomain modeling of the heart,  we refer to \cite{colli02,colli05,colli12}, \cite{henri} for more details. In the sequel, the space-time set $(0,T)\times O$ is denoted by $O_T$ in order to simplify the notation.
\paragraph{\textbf{Basic equations.}} The basic equations modeling the electrical activity of the heart can be obtained as follows. First, we know that the structure of the cardiac tissue can be viewed as composed by two volumes: the intracellular space  $\Omega_i$ (inside the cells) and the extracellular space $\Omega_e$ (outside) separated by the active membrane $\Gamma^{y}$.
  
  Thus, the membrane $\Gamma^{y}$ is pierced by proteins whose role is to ensure ionic transport between the two media (intracellular and extracellular) through this membrane. So, this transport creates an electric current.\\ So by using Ohm's law, the intracellular and extracellular electrical potentials $u_{j}: \Omega_{j,T} \rightarrow \R$ are related to current flows $J_{j}: \Omega_{j,T} \rightarrow \R^d$
\begin{equation*}
J_{j}=\mathrm{M}_{j}\nabla u_{j}, \ \text{in} \ \Omega_{j,T}:=(0,T)\times\Omega_{j},
\end{equation*}
where $\mathrm{M}_{j}$ represents the corresponding conductivity of the tissue for $j=i,e$.\\
In addition, the \textit{transmembrane} potential $v$ is known as the  potential at the membrane $\Gamma^{y}$   which is defined as follows:
\begin{equation*}
v=(u_{i}-u_{e})\vert_{\Gamma^{y}} : (0,T)\times \Gamma^{y} \rightarrow \R.
\end{equation*}
 
  Moreover, we assume the intracellular and extracellular spaces are source-free and thus the intracellular and extracellular potentials $u_{i}$ and $u_{e}$ are solutions to the elliptic equations:
\begin{equation}
\begin{aligned}
&-\text{div}J_{j}=0 \ \text{in} \ \Omega_{j,T}, \text{ for } j=i,e.
\end{aligned}
\label{pb}
\end{equation}
 
 According to the  current conservation law, the transmembrane current $\I_{m}$ is now introduced:
\begin{equation}
\I_m=-J_{i}\cdot n_i=J_{e}\cdot n_e, \ \text{on} \ \Gamma^{y}_{T}:=(0,T)\times \Gamma^{y},
\label{cond_bord} 
\end{equation}
with $n_i$ denotes the unit exterior normal to the boundary $\Gamma_\e$ from intracellular to extracellular domain and $n_e=-n_i$.

 The membrane has both a capacitive property schematized by a capacitor and a resistive property schematized by a resistor. On the one hand, the capacitive property depends on the formation of the membrane which can be represented by a capacitor of capacitance  $C_m$ (the capacity per unit area of the membrane). We recall that the quantity of the charge of a capacitor is $q=C_m v$. Recall that, the capacitive current $\I_c$ is the amount of charge that flows per unit of time:
\begin{equation*}
\I_{c}=\pt_t q=C_m\pt_t v.
\end{equation*}
On the other hand, the resistive property depends on the ionic transport between the intracellular and extracellular media. The resistive current $\I_r$ is defined by the ionic current $\I_{ion}$ measured from the intracellular to the extracellular medium which depends on the transmembrane potential $v$ and the gating variable $w : \Gamma^{y} \rightarrow \R$. Moreover, the total transmembrane current $\I_m$ (see \cite{colli12}) is given by
\begin{equation*}
\I_m=\I_c+\I_r-\I_{app} \text{ on } \Gamma^{y}_{T},
\end{equation*} 
with $\I_{app}$ is the applied current per unit area of the membrane surface.\\
Consequently, the transmembrane potential $v$ satisfies the following dynamic condition on $\Gamma^{y}$ involving the gating variable $w : \Gamma^{y} \rightarrow \R$:
\begin{equation}
\begin{aligned}
&\I_m=  C_m\pt_t v+\I_{ion}(v,w)-\I_{app}  &\text{ on } \Gamma^{y}_{T},
\\ &\pt_t w-H(v,w)=0 &\text{ on } \Gamma^{y}_{T}.
\label{cond_dyn}
\end{aligned}
\end{equation}
Herein, the functions $H$ and $\I_{ion}$ correspond to an ionic model of membrane dynamics.
 
 In addition, we assume that the no-flux boundary condition at the interface $\Gamma^{z}$ is given by:
\begin{equation}
\mathrm{M}_{i}\nabla u_{i}\cdot n_{z}= 0 \quad \text{ on } \Gamma^{z}_{T}:=(0,T)\times \Gamma^z,
\label{cond_dyn_z}
\end{equation}
with $n_{z}$ denotes the unit exterior normal to the boundary $\Gamma^{z}$.\\

\paragraph{\textbf{Statement of the mathematical model and main results.}} $\,$ Cardiac tissues have a number of important inhomogeneities, particularly those related to intercellular communications. The dimensionless analysis done correctly makes the problem simpler and clearer \cite{henri,colli12}. So, we can convert system \eqref{pb}-\eqref{cond_dyn_z} to the following non-dimensional form:
\begin{subequations}
\begin{align}
-\nabla_{\widehat{x}}\cdot\left( \widehat{\mathrm{M}}_{i}^{\e,\delta}\nabla_{\widehat{x}} \widehat{u}_{i}^{\e,\delta}\right)  &=0 &\text{ in } \Omega_{i,T}^{\e,\delta}:=(0,T)\times\Omega_{i}^{\e,\delta}, 
\label{bid_intra}
\\ -\nabla_{\widehat{x}}\cdot\left( \widehat{\mathrm{M}}_{e}^{\e}\nabla_{\widehat{x}} \widehat{u}_{e}^{\e} \right)  &=0 &\text{ in } \Omega_{e,T}^{\e}:=(0,T)\times\Omega_{e}^\e,
\label{bid_extra}
\\ \e\left( \pt_{\widehat{t}} \widehat{v}_\e+\widehat{\I}_{ion}(\widehat{v}_\e,\widehat{w}_\e)-\widehat{\I}_{app,\e}\right) &=\widehat{\I}_m &\ \text{on} \ \Gamma_{\e,T}:=(0,T)\times\Gamma_{\e},
\label{bid_onface}
\\-\widehat{\mathrm{M}}_{i}^{\e,\delta}\nabla_{\widehat{x}} \widehat{u}_{i}^{\e,\delta} \cdot n_i=\widehat{\mathrm{M}}_{e}^{\e}\nabla_{\widehat{x}} \widehat{u}_{e}^{\e} \cdot n_e & =\widehat{\I}_m &\ \text{on} \ \Gamma_{\e,T},
\label{bid_mesocont}
\\ \pt_{\widehat{t}} \widehat{w}_\e-\widehat{H}(\widehat{v}_\e,\widehat{w}_\e) &=0 & \text{ on }  \Gamma_{\e,T},
\label{bid_dyn}
\\ \widehat{\mathrm{M}}_{i}^{\e,\delta}\nabla_{\widehat{x}} \widehat{u}_{i}^{\e,\delta} \cdot n_{z} & =0 &\ \text{on} \ \Gamma_{\delta,T}.
\label{bid_microcont}
\end{align}
\label{pbscale}
\end{subequations}
Note that each equation corresponds to the following sense:
\eqref{bid_intra} Intra quasi-stationary conduction, \eqref{bid_extra} Extra quasi-stationary conduction, \eqref{bid_onface} Reaction surface condition, \eqref{bid_mesocont} Meso-continuity equation, \eqref{bid_dyn} Dynamic coupling, \eqref{bid_microcont} Micro-boundary condition.
 
 For convenience, the superscript $ \  \widehat{\cdot} \ $ of the dimensionless variables is omitted. 
Observe that the bidomain equations are invariant with respect to the scaling parameters $\e$ and $\delta$. Now we define the rescaled electrical potential as follows:
$$u_{i}^{\e,\delta}(t,x):=u_{i}\left( t, x,\frac{x}{\e},\frac{x}{\e\delta}\right), \quad u_{e}^{\e}(t,x):= u_{e}\left( t, x,\frac{x}{\e}\right).$$

Analogously, we obtain the rescaled transmembrane potential  $v_\e=(u_{i}^{\e,\delta}-u_{e}^{\e}){\vert_{\Gamma_{\e,T}}}$ and gating variable $w_\e.$ In general, the functions $v_\e$ and $w_\e$ does not depend on $\delta,$ we omit the index $\delta$  when non confusion arises. Next, we define also the following rescaled symmetric Lipschitz continuous conductivity matrices:  \begin{equation}
\mathrm{M}_{i}^{\e,\delta}(x):=\mathrm{M}_{i}\left( x,\frac{x}{\e},\frac{x}{\e\delta}\right) \ \text{ and } \ \mathrm{M}_{e}^{\e}(x):= \mathrm{M}_{e}\left( x,\frac{x}{\e}\right),
\label{M_ie}
\end{equation}
satisfying the elliptic and periodicity conditions: there exist constants $\alpha, \beta \in \R,$ such that $0<\alpha<\beta$ and for all $\lambda\in \R^d:$
\begin{subequations}
\begin{align}
&\mathrm{M}_j\lambda\cdot\lambda \geq \alpha\abs{\lambda}^2, 
\\& \abs{\mathrm{M}_j\lambda}\leq \beta \abs{\lambda}, \text{ for }  j=i,e,
\\& \mathrm{M}_i \ \mathbf{y}\text{- and }\mathbf{z}\text{-periodic}, \quad \mathrm{M}_e \ \mathbf{y}\text{-periodic}.
\end{align}
\label{A_M_ie}
\end{subequations}
We complete system \eqref{pbscale} with no-flux boundary conditions: 
\begin{equation*}
\left( \mathrm{M}_{i}^{\e,\delta}\nabla u^{\e,\delta}_{i}\right) \cdot \mathbf{n}=\left( \mathrm{M}_{e}^{\e}\nabla u_{e}^{\e}\right) \cdot\mathbf{n}=0 \ \text{ on } \  (0,T)\times \pt_{\text{ext}} \Omega,
\end{equation*}
where $\mathbf{n}$ is the outward unit normal to the exterior boundary of $\Omega$. We impose initial conditions on the transmembrane potential and the gating variable:

\begin{equation}
v_\e(0,x)=v_{0,\e}(x)\quad\text{and}\quad w_\e(0,x)=w_{0,\e}(x) \quad \text{ a.e. on} \ \Gamma_{\e}.
\label{cond_ini_vw}
\end{equation}
We mention for instance some assumptions on the ionic functions, the source term and the initial data:\\
\textbf{Assumptions on the ionic functions.} The ionic current $\I_{ion}(v,w)$ can be decomposed into $I_{1,ion}(v): \R \rightarrow \R$ and $I_{2,ion}(w); \R \rightarrow \R$, where $\I_{ion}(v,w)=I_{1,ion}(v)+I_{2,ion}(w)$. Furthermore, $I_{1,ion}$ is considered as a $C^1$ function, $I_{2,ion}$ and $H : \R^2 \rightarrow \R$ are linear functions. Also, we assume that there exists $r\in (2,+\infty)$ and constants $\alpha_1,\alpha_2,\alpha_3, \alpha_4,\alpha_5, C>0$ and $\beta_1, \beta_2>0$ such that:
\begin{subequations}
\begin{align}
&\dfrac{1}{\alpha_1} \abs{v}^{r-1}\leq \abs{I_{1,ion}(v)}\leq \alpha_1\left(  \abs{v}^{r-1}+1\right), \,\abs{I_{2,ion}(w)}\leq \alpha_2(\abs{w}+1), 
\\& \abs{H(v,w)}\leq \alpha_3(\abs{v}+\abs{w}+1),\text{ and }
    I_{2,ion}(w)v-\alpha_4H(v,w)w\geq \alpha_5 \abs{w}^2,
\\& \tilde{I}_{1,ion} : z\mapsto I_{1,ion}(z)+\beta_1 z+\beta_2 \text{ is strictly increasing with } \lim \limits_{z\rightarrow 0} \tilde{I}_{1,ion}(z)/z=0,
\\& \forall z_1,z_2 \in \R,\,\,\left(\tilde{I}_{1,ion}(z_1)-\tilde{I}_{1,ion}(z_2) \right)(z_1-z_2)\geq \dfrac{1}{C} \left(1+\abs{z_1}+\abs{z_2} \right)^{r-2} \abs{z_1-z_2}^{2}. 
\end{align}
\label{A_H_I}
\end{subequations}
\begin{rem} Few models to these functions are available, we mention for  instance the Hodgkin-Huxley model \cite{hodgkin}, the Mitchell-Schaeffer model \cite{mitchell}, the Roger-McCulloch model \cite{roger} and the Aliev-Panfilov model \cite{aliev}. Here, we take the Fitzhugh-Nagumo model \cite{fitz,nagumo} which is defined as follows
  \begin{subequations}
  \begin{align}
  H(v,w)&= av-bw, \\  \I_{ion}(v,w)&=\left( \lambda v (1-v)(v-\theta) \right) + (-\lambda w):= I_{1,ion}(v)+I_{2,ion}(w)
  \end{align}
  \label{ionic_model}
  \end{subequations}
  where $a, b, \lambda, \theta$ are given parameters with $a,b\geq 0, \ \lambda<0$ and $0<\theta<1.$
\end{rem}
\textbf{Assumptions on the source term.} There exists a constant $C>0$ independent of $\e$ such that the source term $\I_{app,\e}$ satisfies the following estimation:
\begin{equation}
\norm{\e^{1/2}\I_{app,\e}}_{L^{2}(\Gamma_{\e,T})}\leq C,
\label{A_iapp}
\end{equation}
where $\Gamma_{\e,T}:=(0,T)\times\Gamma_{\e}.$

 \textbf{Assumptions on the initial data.} The initial conditions $v_{0,\e}$ and $w_{0,\e}$ satisfy the following estimation:
\begin{equation}
\norm{\e^{1/r}v_{0,\e}}_{L^{r}(\Gamma_{\e})}+\norm{\e^{1/2}v_{0,\e}}_{L^{2}(\Gamma_{\e})}+\norm{\e^{1/2}w_{0,\e}}_{L^{2}(\Gamma_{\e})}\leq C,
\label{A_vw0}
\end{equation}
for some constant $C$ independent of $\e.$ Moreover, $v_{0,\e}$ and $w_{0,\e}$ are assumed to be traces of uniformly bounded sequences in $C^{1}(\overline{\Omega}).$ 

Clearly, the equations in \eqref{pbscale} are invariant under the simultaneous change of $u_i^{\e,\delta}$ and $u_{e}^{\e}$ into $u_i^{\e,\delta}+k;$ $u_{e}^{\e}+k,$ for any $k\in \R.$ Hence, we may impose the following normalization condition:
\begin{equation}
\int_{\Omega_{e}^{\e}}u_{e}^{\e}(t,x)dx=0 \text{ for a.e. } t\in(0,T).
\label{normalization_cond} 
\end{equation}
 
   We start by stating the weak formulation of the microscopic bidomain model as given in the following definition.
  
\begin{defi}[Weak formulation] A weak solution of problem \eqref{pbscale}-\eqref{cond_ini_vw} is a four tuple \\ $(u_{i}^{\e,\delta},u_{e}^{\e},v_{\e},w_{\e})$ such that $u_{i}^{\e,\delta}\in L^{2}\left(0,T;H^{1}\left( \Omega_{i}^{\e,\delta}\right)\right),$ $u_{e}^{\e}\in L^{2}\left(0,T;H^{1}(\Omega_{e}^{\e})\right) ,$ $v_{\e}$ $\in L^{2}(0,T;H^{1/2}(\Gamma_{\e}))$ $\cap L^{r}(\Gamma_{\e,T}),$ $r \in (2,+\infty),$ $w_\e \in L^{2}(\Gamma_{\e,T}),$  $\pt_t v_\e \in L^{2}(0,T;H^{-1/2}(\Gamma_{\e}))$ $+ L^{r/(r-1)}(\Gamma_{\e,T}),$ $\pt_t w_\e \in L^{2}(\Gamma_{\e,T})$ and satisfying the following weak formulation for a.e. $t\in(0,T):$ 
\begin{equation}
\begin{aligned}
\iint_{\Gamma_{\e,T}} \e\pt_t v_\e\varphi \ d\sigma_xdt &+\iint_{\Omega_{i,T}^{\e,\delta}}\mathrm{M}_{i}^{\e,\delta}(x)\nabla u_{i}^{\e,\delta}\cdot\nabla\varphi_i \ dxdt
+\iint_{\Omega_{e,T}^{\e}}\mathrm{M}_{e}^{\e}(x)\nabla u_{e}^{\e}\cdot\nabla\varphi_e \ dxdt
\\& +\iint_{\Gamma_{\e,T}} \e\I_{ion}(v_\e,w_\e)\varphi\ d\sigma_xdt=\iint_{\Gamma_{\e,T}} \e\I_{app,\e}\varphi\ d\sigma_xdt,
\end{aligned}
\label{Fv_i_ini}
\end{equation}
\begin{equation}
\iint_{\Gamma_{\e,T}} \pt_t w_\e\phi \ d\sigma_xdt-\iint_{\Gamma_{\e,T}} H(v_\e,w_\e)\phi \ d\sigma_xdt=0,
\label{Fv_d_ini}
\end{equation}
for all $\varphi_i \in L^{2}\left(0,T;H^{1}\left( \Omega_{i}^{\e,\delta}\right)\right), \ \varphi_e \in L^{2}\left(0,T;H^{1}\left( \Omega_{e}^{\e}\right)\right)$ with $\varphi=\left(\varphi_{i}-\varphi_{e}\right)\vert_{\Gamma_{\e,T}}$ $\in$ $L^{2}\left(0,T;H^{1/2}(\Gamma_{\e})\right)$ $\cap L^{r}(\Gamma_{\e,T})$ and $\phi \in L^{2}(\Gamma_{\e,T})$. Moreover, the weak formulation makes sense in view of the following initial conditions:
\begin{equation}
v_\e(0,x)=v_{0,\e}(x)\quad\text{and}\quad w_\e(0,x)=w_{0,\e}(x) \quad \text{ a.e. on} \ \Gamma_{\e}.
\label{cond_ini_vw_weak}
\end{equation}
\label{Fv} 
\end{defi}

 Then, the existence of the weak solution is given in the following theorem with the proof can be found in \cite{bendunf19} where the mesoscopic domain is ignored.
\begin{thm}[Microscopic Bidomain Model] Assume that the conditions \eqref{M_ie}-\eqref{normalization_cond} hold. Then the microscopic bidomain problem \eqref{pbscale}-\eqref{cond_ini_vw} possesses a unique weak solution in the sense of Definition \ref{Fv} for every fixed $\e,$ $\delta >0$.
Moreover, there exists a constant $C>0$ not depending on $\e$ and $\delta$ such that:
\begin{equation}
\norm{\sqrt{\e}v_\e}_{L^{\infty}(0,T;L^2(\Gamma_\e))}+\norm{\sqrt{\e}w_\e}_{L^{\infty}(0,T;L^2(\Gamma_\e))} \leq C,
\label{E_vw}
\end{equation}
\begin{equation}
\norm{u_{i}^{\e,\delta}}_{L^{2}\left(0,T;H^{1}\left( \Omega_{i}^{\e,\delta}\right) \right)}\leq C , \quad \norm{u_{e}^{\e}}_{L^{2}\left(0,T;H^{1}\left(\Omega_{e}^{\e}\right) \right)}\leq C,
\label{E_u}
\end{equation}
\begin{equation}
\norm{\e^{1/r}v_\e}_{L^{r}(\Gamma_{\e,T})}\leq C \quad \text{and}\quad \norm{\e^{(r-1)/r}\; \mathrm{I}_{1,ion}(v_\e)}_{L^{r/(r-1)}(\Gamma_{\e,T})}\leq C.
\label{E_vr}
\end{equation}
Furthermore, if $v_{0,\e} \in H^{1/2}(\Gamma_{\e})\cap L^{r}(\Gamma_\e)$, there exists a constant $C>0$ not depending on $\e$ and $\delta$ such that:
\begin{equation}
\norm{\sqrt{\e}\pt_t v_\e}_{L^{2}(\Gamma_{\e,T})}+\norm{\sqrt{\e}\pt_t w_\e}_{L^2(\Gamma_{\e,T})}\leq C.
\label{E_dtvw}
\end{equation}
\label{thm_micro}
\end{thm}

The existence and uniqueness of weak solutions for the microscopic bidomain problem \eqref{pbscale}-\eqref{cond_ini_vw}  for every fixed $\e,$ $\delta >0$ is standard, e.g., by using the Faedo-Galerkin method based on a priori estimates \eqref{E_vw}-\eqref{E_dtvw}. We notice that we get the same energy estimates as in \cite{bendunf19}, this comes from the consideration of homogeneous Neumann type conditions on the microscopic scale.

\paragraph{\textbf{Main results.}} In this part, we highlight the main results obtained in our paper. Based on a priori estimates and the convergence results of unfolding homogenization method, we can pass to the limit in the microscopic equations and derive the following macroscopic problem: 
\begin{thm}[Macroscopic Bidomain Model] 
A sequence of solutions $\Bigl((u_{i}^{\e,\delta})_{\e,\delta},(u_{e,\e})_\e,(w_{\e})_{\e}\Bigl)$ of the microscopic bidomain model \eqref{pbscale}-\eqref{cond_ini_vw} converges $($as $\e,\delta \to 0)$ to a weak solution $(u_i,u_e,w)$ with $v=u_i-u_e$, $u_i,u_e\in L^2(0,T;H^1(\Omega))$, $v \in L^2(0,T;H^1(\Omega))\cap L^r(\Omega_T)$, $\partial_t v\in L^2(0,T;H^{-1}(\Omega))+ L^{r/(r-1)}(\Omega_T)$ and $w\in C(0,T;L^2(\Omega))$, of the macroscopic problem
\begin{equation}
\begin{aligned}
\mu_{m}\pt_{t} v+\nabla \cdot\left(\widetilde{\mathbf{M}}_{e}\nabla u_{e}\right) +\mu_{m}\I_{ion}(v,w) &= \mu_{m}\I_{app} &\text{ in } \Omega_{T},
\\ \mu_{m}\pt_{t} v-\nabla \cdot\left( \doublewidetilde{\mathbf{M}}_{i}\nabla u_{i}\right)+\mu_{m}\I_{ion}(v,w) &= \mu_{m} \I_{app} &\text{ in } \Omega_{T},
\\ \pt_{t} w-H(v,w) &=0 & \text{ on }  \Omega_{T},
\label{pb_macro}
\end{aligned}
\end{equation}
completed with no-flux boundary conditions on $u_i, u_e$ on $\pt_{\text{ext}} \Omega:$ 
\begin{equation*}
\left(\widetilde{\mathbf{M}}_{e}\nabla u_{e}\right)\cdot\mathbf{n}=\left( \doublewidetilde{\mathbf{M}}_{i}\nabla u_{i}\right) \cdot\mathbf{n}=0 \ \text{ on } \  \Sigma_T:=(0,T)\times\pt_{\text{ext}} \Omega,
\end{equation*}
and initial conditions for the transmembrane potential $v$ and the gating variable $w :$ 
\begin{equation}
v(0,x)=v_0(x)\qquad\text{and}\qquad w(0,x)=w_0(x) \text{ a.e. on }  \Omega,
\label{cond_ini}
\end{equation}
where $v_{0}, w_{0} \in L^2(\Omega)$ are defined in Remark \ref{conv_trace_ini_norm}-\ref{conv_cond_ini}. Here, $\mu_{m}=\abs{\Gamma^{y}}/\abs{Y}$ is the ration between the surface membrane and the volume of the reference cell. Moreover, $\mathbf{n}$ is the outward unit normal to the exterior boundary of $\Omega.$ The first-level homogenized conductivity matrices $\widetilde{\mathbf{M}}_j=\left( \widetilde{\mathbf{m}}^{pq}_j\right)_{1\leq p,q \leq d}$ for $j=i,e$ and the second-level one $\doublewidetilde{\mathbf{M}}_i=\left( \doublewidetilde{\mathbf{m}}^{pq}_i\right)_{1\leq p,q \leq d}$ are respectively defined by:
\begin{subequations}
\begin{align}
&\widetilde{\mathbf{m}}^{pq}_e:=\dfrac{1}{\abs{Y}}\overset{d}{\underset{k=1}{\sum}}\displaystyle\int_{Y_e}\left( \mathrm{m}_e^{pq}+\mathrm{m}^{pk}_{e}\dfrac{\pt \chi_e^q}{\pt y_k}\right) \ dy,\quad  \widetilde{\mathbf{m}}^{pq}_i:=\dfrac{1}{\abs{Z}}\overset{d}{\underset{\ell=1}{\sum}}\displaystyle\int_{Z_{c}}\left( \mathrm{m}_i^{pq}+\mathrm{m}^{p\ell}_{i}\dfrac{\pt \theta_i^q}{\pt z_\ell}\right) \ dz,
\label{tilde_m_i}
\end{align}
\begin{equation}
\begin{aligned}
& \doublewidetilde{\mathbf{m}}^{pq}_i:=\dfrac{1}{\abs{Y}}\overset{d}{\underset{k=1}{\sum}}\displaystyle \int_{Y_{i}}  \left(\widetilde{\mathbf{m}}_i^{pk} \dfrac{\pt \chi_{i}^q}{\pt y_k}(y)+\widetilde{\mathbf{m}}_i^{pq}\right) dy
\\& \quad \quad = \dfrac{1}{\abs{Y}} \dfrac{1}{\abs{Z}}\overset{d}{\underset{k,\ell=1}{\sum}}\displaystyle \int_{Y_{i}} \int_{Z_{c}} \left[\displaystyle \left(\mathrm{m}^{pk}_{i}+ \mathrm{m}^{p\ell}_{i}\dfrac{\pt \theta_{i}^{k}}{\pt z_\ell}\right) \dfrac{\pt \chi_{i}^q}{\pt y_k}(y)+\left(\mathrm{m}^{pq}_{i}+ \mathrm{m}^{p\ell}_{i}\dfrac{\pt \theta_{i}^{q}}{\pt z_\ell}\right)\right] \ dzdy,
\end{aligned}
\label{titilde_m_i}
\end{equation}
\end{subequations}
where the components $\chi_{e}^{q}$ of $\chi_e$ and $\chi_{i}^{q}$  of $\chi_i$ are respectively the corrector functions, solutions of the $\e$-cell problems:
\begin{subequations}
\begin{equation}
\begin{cases}
-\nabla_{y}\cdot\left( \mathrm{M}_{e}\nabla_{y} \chi_{e}^q\right) =\nabla_{y}\cdot\left( \mathrm{M}_{e} e_{q}\right) \ \text{in} \ Y_{e},
\\ \chi_e^q \ y\text{-periodic}, 
\\ \mathrm{M}_{e} \nabla_y \chi_e^q \cdot n_e= - (\mathrm{M}_{e}e_q )\cdot n_e \text{ on } \Gamma^{y},
\end{cases}
\end{equation}
\begin{equation}
\begin{cases}
-\nabla_{y}\cdot\left( \widetilde{\mathbf{M}}_{i}\nabla_{y} \chi_{i}^q\right) =\nabla_{y}\cdot\left( \widetilde{\mathbf{M}}_{i} e_{q}\right) \ \text{in} \ Y_{i},
\\ \chi_{i}^{q} \ y\text{-periodic},
\\ \widetilde{\mathbf{M}}_{i} \nabla_y \chi_{i}^q \cdot n_{i}=-\left(\widetilde{\mathbf{M}}_{i} e_q \right)\cdot n_{i} \ \text{on} \ \Gamma^{y},  
 \end{cases}
 \end{equation}
 \end{subequations}
and the component $\theta_{i}^{q}$ of $\theta_{i}$ is the corrector function, solution of the $\delta$-cell problem:
\begin{equation}
\begin{cases}
\nabla_{z}\cdot\left( \mathrm{M}_{i}\nabla_{z} {\theta}_{i}^q\right)=\nabla_{z}\cdot\left( \mathrm{M}_{i}e_{q}\right) \ \text{in} \ Z_{c},\\ \theta_{i}^q \ y\text{- and }z\text{-periodic}, \\ \mathrm{M}_{i}\nabla_z \theta_{i}^q \cdot n_{z}=-(\mathrm{M}_{i}e_q)\cdot n_{z} \ \text{on} \ \Gamma^{z}, 
\end{cases}
\end{equation}
for $e_q$, $q=1,\dots,d,$ the standard canonical basis in $\R^d.$ System  \eqref{pb_macro}-\eqref{cond_ini} corresponds to the sought macroscopic equations. Finally, note that we close the problem by the normalization condition on the extracellular potential for almost all $t\in [0,T],$
\begin{equation*}
\int_{\Omega} u_{e} (t,x)dx=0.
\end{equation*}
\label{thm_macro}
\end{thm}

The proof of Theorem \ref{thm_macro} is given in Section \ref{unf}.
The uniqueness of the solutions to the macroscopic model can be proved by standard methods. This implies that all the convergence results remain valid for the whole sequence. It is easy to verify that the macroscopic conductivity tensors of the intracellular $\doublewidetilde{\mathbf{M}}_{i}$ and extracellular $\widetilde{\mathbf{M}}_{e}$ spaces are symmetric and positive definite.

\begin{rem} The authors in \cite{bendunf19} treated the initial problem with the coefficients $\mathrm{m}^{pq}_{j}$ depending only on the variable $y$ for $j=i,e$. Comparing to \cite{bendunf19}, in our work the microscopic conductivity matrix $\mathrm{M}_{i}$ of the intracellular space depends on two variables $y$ and $z$. Using a three-scale unfolding method, we derive a new approach of the homogenized model \eqref{pb_macro} from the microscopic problem \eqref{pbscale}. Our homogenized problem is described in three steps. First, we unfold the weak formulation of the initial problem and prove the convergence results of the corresponding terms using the properties of the unfolding operators. Next, we pass to the limit in the unfolded formulation and we find the explicit forms of the associated solutions. Finally, the last step describes the two-level homogenization whose the homogenized (macroscopic) conductivity matrix $\doublewidetilde{\mathbf{M}}_i$ of the intracellular space  are integrated with respect to $z$ and then with respect to $y$.
\end{rem}

\section{The unfolding method in perforated domains}\label{methodunf} In this section, we give the definitions for the concepts of unfolding operator defined on the domain $\Omega_T$ and on the membrane $\Gamma^{y}_{T}.$ Further, we recall some properties and results related to these concepts used in our paper.

\subsection{Unfolding Operator}For the reader's convenience, we recall the notion of the unfolding operator. The following results can be found in \cite{doinaunf12}.
\subsubsection{\textbf{Definition of the unfolding operator}}\label{domain_unf}	
 In order to define an unfolding operator, we first introduce the following sets in $\R^d$ (see Figure \ref{fig_unf})
\\ 
 \begin{itemize}
 \item $\Xi_\e=\lbrace k \in \mathbb{Z}^d, \ \e(k_{\ell}+Y)\subset\Omega\rbrace,$ where $ k_\ell:=( k_1\ell^\text{mes}_1,\dots,  k_d \ell^\text{mes}_d ),$
 \\
 \item $\Xi_\delta=\lbrace k' \in \mathbb{Z}^d, \ \delta(k'_{\ell'}+Z)\subset\Omega\rbrace,$ where $ k'_{\ell'}=( k'_1\ell^\text{mic}_1,\dots,  k'_d \ell^\text{mic}_d ),$
 \\
 \item $\widehat{\Omega}^{\e}=$ interior $\lbrace\underset{k \in \Xi_\e }{\bigcup} \e (k_{\ell} +\overline{Y})\rbrace,$
 \\
 \item $\widehat{\Omega}_{e}^{\e}= $ interior $\lbrace\underset{k \in \Xi_\e }{\bigcup} \e (k_{\ell} +\overline{Y_e})\rbrace,$
 \\
 \item $\widehat{\Omega}_{i}^{\e,\delta}=$ interior $\lbrace\underset{k \in \Xi_\e }{\bigcup} \e (k_{\ell} +\overline{Y_i}^\delta)\rbrace, \qquad \overline{Y_i}^\delta=$ interior $\lbrace \underset{k' \in \Xi_\delta }{\bigcup} \delta(k'_{\ell'} + \overline{Z_{c}})\rbrace,$ 
 \\
  \item $\widehat{\Gamma}_\e=\lbrace y\in \Gamma^{y} : y\in \widehat{\Omega}^{\e}\rbrace,$
 \\
 \item $\Lambda^\e=\Omega\setminus\widehat{\Omega}^{\e},$
 \\
 \item $\widehat{\Omega}_{T}^{\e}=(0,T)\times \widehat{\Omega}^{\e},$
 \\
 \item $\widehat{\Omega}_{i,T}^{\e,\delta}=(0,T)\times \widehat{\Omega}_{i}^{\e,\delta}, \qquad \widehat{\Omega}_{e,T}^{\e}=(0,T)\times \widehat{\Omega}_{e}^{\e}, $
 \\
 \item $\Lambda_{T}^\e=(0,T)\times \Lambda^\e.$
 \end{itemize}
 \begin{figure}[h!]
  \centering
  \includegraphics[width=11cm]{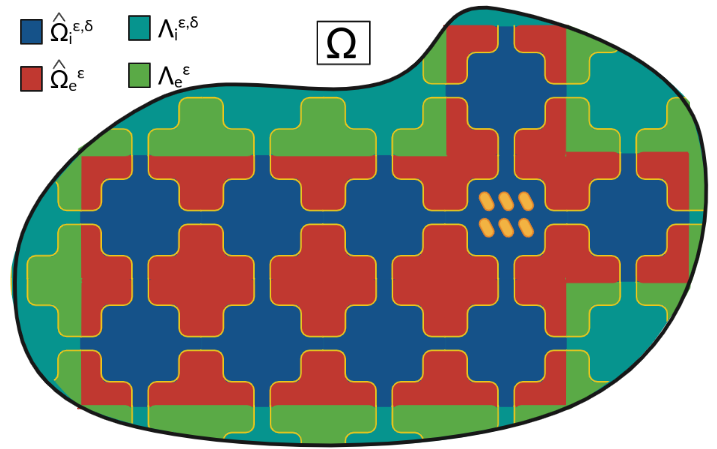}
  \caption{The sets $\widehat{\Omega}_{i}^{\e,\delta}$ (in blue), $\widehat{\Omega}_{e}^{\e}$ (in red), $\Lambda_i^{\e,\delta}$ (in white)  and $\Lambda_e^\e$ (in green).}
  \label{fig_unf}
 \end{figure} 
 For all $w\in \R^d,$ let $[w]_{Y}$ be the unique integer combination of the periods such that $w-[w]_{Y} \in Y.$ We may write $w=[w]_{Y}+\lbrace w\rbrace_Y$ for all $w\in \R^d,$ so that for all $\e>0,$ we get the unique decomposition:
 $$x=\e\left(\left[ \dfrac{x}{\e} \right]_Y + \left\lbrace \dfrac{x}{\e} \right\rbrace_Y  \right), \ \text{for all } x\in \R^d.  $$
Based on this decomposition, we define the unfolding operator in perforated domains.
\begin{defi} For any function $\phi$ Lebesgue-measurable on $(0,T)\times\Omega_{i}^{\e}$ (the intracellular medium at mesoscale), the unfolding operator $\T_{\e}^{i}$ is defined as follows:
\begin{equation}
\T_{\e}^{i}(\phi)(t,x,y)=
\begin{cases}
\phi\left(t, \e\left[ \dfrac{x}{\e} \right]_Y +\e y \right) &\text{ a.e. for } (t, x, y) \in \widehat{\Omega}^{\e}_{T}\times Y_{i}, \\ 0  &\text{ a.e. for } (t, x, y) \in \Lambda^{\e}_{T}\times Y_{i},
\end{cases}
\label{op_ie}
\end{equation}
where $\left[ \cdot\right] $ denotes the Gau$\beta$-bracket.
Similarly, we define the unfolding operator $\T_{\e}^{e}$ on the domain $(0,T)\times\Omega_{e}^{\e}.$
\end{defi}

We readily have that:
$$\forall x\in \Omega_{i}^{\e}\cap\widehat{\Omega}^{\e}, \ \T_{\e}^{i}(\phi)\left(t, x, \left\lbrace \dfrac{x}{\e} \right\rbrace_Y  \right)=\phi(t,x). $$

\subsubsection{\textbf{Properties of the unfolding operator}}
The following results summarizes some basic properties of the unfolding operator and we refer the reader to \cite{doinaunf06,doinaunf12} for more details.
\begin{prop}
For $p\in [1,+\infty),$ the operator $\T_{\e}^{i}$ is linear and continuous from $L^{p}\left((0,T) \times\Omega_{i}^{\e}\right)$ to $L^{p}(\Omega_T\times Y_{i}).$ For every $\phi \in L^{1}\left((0,T) \times\Omega_{i}^{\e}\right)$ and $u,v \in L^{p}\left((0,T) \times\Omega_{i}^{\e}\right),$ the following formula holds:
\begin{enumerate}
\item[$1.$]\label{Prop1_uo} $\T_{\e}^{i}(uv)=\T_{\e}^{i}(u)\T_{\e}^{i}(v),$
\\
\item[$2.$]\label{Prop2_uo} $\dfrac{1}{\abs{Y}}\displaystyle\iint_{\Omega\times Y_i} \T_{\e}^{i}(\phi)(t,x,y) \ dx dy=\int_{\widehat{\Omega}_{i}^{\e,\delta}} \phi(t, x) \ dx=\int_{\Omega_{i}^{\e}} \phi(t, x) \ dx-\int_{\Lambda_{i}^{\e}} \phi(t, x) \ dx,$ 
\\
\item[$3.$] \label{Prop3_uo} $\norm{\T_{\e}^{i}(u)}_{L^{p}(\Omega\times Y_i)}=\abs{Y}^{1/p}\norm{u \mathds{1}_{\widehat{\Omega}_{i}^{\e,\delta}}}_{L^{p}\left( \Omega_{i}^{\e}\right)}\leq \abs{Y}^{1/p}\norm{u}_{L^{p}\left( \Omega_{i}^{\e}\right)},$
\\
\item[$4.$] \label{Prop4_uo} $\nabla_y\T_{\e}^{i}(u)(t,x,y)=\e\T_{\e}^{i}(\nabla_x u)(t,x,y).$
\end{enumerate}
\label{prop_uo}
\end{prop}
In the sequel, we will define $W_{\#}^{1,p}$ the periodic Sobolev space as follows
\begin{defi} Let $\mathcal{O}$ be a reference cell and $p\in[1,+\infty)$. Then, we define
\begin{equation}
W_{\#}^{1,p}(\mathcal{O})=\lbrace u \in W^{1,p}(\mathcal{O}) \ \text{such that} \ u \ \text{periodic with }  \M_{\mathcal{O}}(u)=0\rbrace,
\label{W}
\end{equation}
where $\M_{\mathcal{O}}(u)=\dfrac{1}{\abs{\mathcal{O}}}\displaystyle\int_{\mathcal{O}} u \ dy.$
Its duality bracket is defined by: $$F(u)=(F,u)_{(W_{\#}^{1,p}(\mathcal{O}))',W_{\#}^{1,p}(\mathcal{O})}=(F,u)_{(W^{1,p}(\mathcal{O}))',W^{1,p}(\mathcal{O})},\ \forall u \in W_{\#}^{1,p}(\mathcal{O}).$$
Furthermore, by the Poincaré-Wirtinger's inequality, the Banach space $W_{\#}^{1,p}$ has the following norm:
$$\norm {u}_{W_{\#}^{1,p}(\mathcal{O})}=\norm {\nabla u}_{L^p(\mathcal{O})}, \forall u \in W_{\#}^{1,p}(\mathcal{O}).$$
{\bfseries \textsc{Notation}:} We denote $W_{\#}^{1,2}(\mathcal{O})$ by $H_{\#}^{1}(\mathcal{O})$  for $p=2.$
\end{defi}
We turn now to the convergence properties for the corresponding unfolding operator, see for e.g. Proposition 2.8 and Theorem 3.12 in \cite{doinaunf12}.
\begin{prop}Let $p\in [1,+\infty).$
\begin{enumerate}
\item[$1.$]For $\phi\in L^{p}(\Omega_T),$ 
$$\T_{\e}^{i}(\phi) \underset{\e\rightarrow 0}{\longrightarrow} \phi \text{ strongly in } L^{p}(\Omega_T\times Y_i).$$
\item[$2.$]Let $\{ \phi^\e \}$ be a sequence in $L^{p}(\Omega_T)$ such that :
$$ \phi^\e \underset{\e\rightarrow 0}{\longrightarrow} \phi \text{ strongly in } L^{p}(\Omega_T)$$
Then,
$$\T_{\e}^{i}(\phi^\e) \underset{\e\rightarrow 0}{\longrightarrow} \phi \text{ strongly in } L^{p}(\Omega_T\times Y_i). $$
\end{enumerate}
\label{T_e_strong}
\end{prop}

\begin{thm}
Let $p\in (1,+\infty).$ Suppose that $u^\e \in L^{p}\left( 0,T; W^{1,p}\left( \Omega_{i}^{\e}\right)\right)$ satisfies
$$ \norm{u^\e}_{L^{p}\left( 0,T; W^{1,p}\left( \Omega_{i}^{\e}\right)\right)}\leq C. $$
Then, there exist $u \in L^{p}\left( 0,T; W^{1,p}(\Omega)\right)$ and $\widehat{u}\in L^{p}\left( 0,T; L^{p}\left(\Omega, W_{\#}^{1,p}(Y_i)\right)\right),$ such that, up to a subsequence, the following hold when $\e\rightarrow 0:$
\begin{enumerate}
\item[$1.$] $\T_{\e}^{i}(u^\e)\rightharpoonup u \text{ weakly in } L^{p}\left( 0,T; L^{p}\left(\Omega, W^{1,p}(Y_i)\right)\right),$
\\
\item[$2.$] $\T_{\e}^{i}(\nabla u^\e) \rightharpoonup \nabla u+\nabla_y \widehat{u} \text{ weakly in } L^{p}(\Omega_T\times Y_i),$
\end{enumerate}
with the space $W_{\#}^{1,p}$ is defined by \eqref{W}.
\label{T_e_weak}
\end{thm}
\subsection{Composition of unfolding operators}
In the intracellular problem, since the electrical potential $u_i^{\e,\delta}$ depends on the mesoscopic variable $y$ and the microscopic one $z$ so we will define a composition of two unfolding operators.  
In this section, we compose unfolding operators with the following convention:\\
Any unfolding operator acts on the two last variables of a function. Herein, we will state the result for a composition of two unfolding operators (see \cite{meunier05}). Let $Y$ and $Z$ be two reference cells (see Figure \ref{three_scale}). For $\e,\delta >0,$ with $ \delta\leq \e,$ the unfolding operators $\T_{\e}^{i}$ and $\T_{\delta}$ are respectively associated to $Y_i$ and $Z_{c}.$ The unfolding operator $\T_{\delta}$ is defined on $\Omega^{\e}_{T}\times Y_{i}$ as follows:
\begin{equation}
\T_{\delta}(\psi)(t,x,y,z)=
\begin{cases}
\psi\left(t, x, \delta\left[ \dfrac{y}{\delta} \right]_Z +\delta z \right) &\text{ a.e. for } (t, x, y, z) \in \widehat{\Omega}^{\e}_{T}\times Y_{i}\times Z_{c}, \\ 0  &\text{ a.e. for } (t, x, y, z) \in \Lambda^{\e}_{T}\times Y_{i}\times Z_{c},
\end{cases}
\label{op_d}
\end{equation}
for any function $\psi$ Lebesgue-measurable on $(0,T)\times\Omega_{i}^{\e,\delta}\times Y_i$.

First, we define the composition of the unfolding operators associated to $Y_i$ and $Z_{c}$ as follows:
\begin{equation}
\T_{\delta}\left(\T_{\e}^{i}(\phi)\right) (t,x,y,z)=
\begin{cases}
\phi\left(t, \e\left[ \dfrac{x}{\e} \right]_Y+\e\delta\left[ \dfrac{y}{\delta} \right]_Z +\e\delta z \right) &\text{ a.e. for } (t, x, y, z) \in \widehat{\Omega}^{\e}_{T}\times Y_{i}\times Z_{c}, \\ 0  &\text{ a.e. for } (t, x, y, z) \in \Lambda^{\e}_{T}\times Y_{i}\times Z_{c},
\end{cases}
\label{op_d_e}
\end{equation}
for any function $\phi$ Lebesgue-measurable on $(0,T)\times\Omega_{i}^{\e,\delta}.$

 We see immediately that for all \textcolor{blue}{$x\in \widehat{\Omega}_{i}^{\e,\delta},$} we have:
$$\T_{\delta}\left(\T_{\e}^{i}(\phi)\right) \left( t,x,\left\lbrace  \dfrac{x}{\e}\right\rbrace_Y,\left\lbrace \dfrac{ \left\lbrace \dfrac{x}{\e}\right\rbrace}{\delta}\right\rbrace_Z \right)=\phi(t,x). $$
Next, we also have some properties for this composition of unfolding operators (see \cite{meunier05} for more details).
\begin{prop}For $p\in [1,+\infty),$ the operator $\T_{\delta}\left(\T_{\e}^{i}(\cdot)\right)$ is linear and continuous from $L^{p}\left((0,T)\times\Omega_{i}^{\e,\delta}\right)$ to $L^{p}(\Omega_T\times Y_{i}\times Z_{c}).$ For every $\phi \in L^{1}\left((0,T)\times\Omega_{i}^{\e,\delta}\right)$ and $u,v \in L^{p}\left((0,T)\times\Omega_{i}^{\e,\delta}\right),$ the following formula holds: \\
\begin{enumerate}
\item\label{Prop1_cuo} $\T_{\delta}\left(\T_{\e}^{i}(uv)\right)=\T_{\delta}\left(\T_{\e}^{i}(u)\right)\T_{\delta}\left(\T_{\e}^{i}(v)\right),$
\\
\item\label{Prop2_cuo} $\dfrac{1}{\abs{Y}}\dfrac{1}{\abs{Z}}\displaystyle\iiint_{\Omega\times Y_i\times Z_{c}} \T_{\delta}\left(\T_{\e}^{i}(\phi)\right)(t,x,y,z) \ dx dydz=\int_{\widehat{\Omega}_{i}^{\e,\delta}} \phi(t, x) \ dx,$
\\
\item\label{Prop3_cuo} $\norm{\T_{\delta}\left(\T_{\e}^{i}(u)\right)}_{L^{p}(\Omega\times Y_ic\times Z_{c})}=\abs{Y}^{1/p}\abs{Z}^{1/p}\norm{u \mathds{1}_{\widehat{\Omega}_{i}^{\e,\delta}}}_{L^{p}\left( \Omega_{i}^{\e,\delta}\right)}\leq \abs{Y}^{1/p}\abs{Z}^{1/p}\norm{u}_{L^{p}\left( \Omega_{i}^{\e,\delta}\right)}.$
\\
\item\label{Prop4_cuo} $\nabla_z\T_{\delta}\left(\T_{\e}^{i}(u)\right)=\delta\T_{\delta}\left(\nabla_y\T_{\e}^{i}(u)\right)=\e\delta\T_{\delta}\left(\T_{\e}^{i}(\nabla_x u)\right).$
\end{enumerate}
\label{prop_cuo}
\end{prop} 

 Using the convergences properties of each unfolding operator, we can prove the following results:
\begin{prop}
For $p\in [1,+\infty).$ Let $\{ \phi^{\e,\delta} \}$ be a sequence in $L^{p}(\Omega_T)$ such that :
$$ \phi^{\e,\delta} \underset{\e \rightarrow 0}{\longrightarrow} \phi \text{ strongly in } L^{p}(\Omega_T)$$
Then,
$$\T_{\delta}\left(\T_{\e}^{i}(\phi^{\e,\delta})\right) \underset{\e\rightarrow 0}{\longrightarrow} \phi \text{ strongly in } L^{p}(\Omega_T\times Y_i \times Z).$$
\label{T_d_e_strong}
\end{prop}

Finally, we end by stating the main convergence result which proved as Theorem 4.1 and Theorem 6.1 in \cite{meunier05} (see also Theorem 5.17 in \cite{doinaunf12}):
\begin{thm} Let $\{ u^{\e,\delta}_{i} \}$ be sequence in $L^{p}\left( 0,T; W^{1,p}\left( \Omega_{i}^{\e,\delta}\right)\right)$ for $p\in (1,+\infty).$ satisfies
$$ \norm{u^{\e,\delta}_{i}}_{L^{p}\left( 0,T; W^{1,p}\left( \Omega_{i}^{\e,\delta}\right)\right)}\leq C. $$
Then, there exist $u_i \in L^{p}\left( 0,T; W^{1,p}(\Omega)\right),\widehat{u}_i\in L^{p}\left( 0,T; L^{p}\left(\Omega, W_{\#}^{1,p}(Y_i)\right)\right)$ and \\ $\widetilde{u}_i\in L^{p}\left( 0,T; L^{p}\left(\Omega \times Y_i, W_{\#}^{1,p}(Z_{c})\right)\right),$ such that, up to a subsequence, the following convergences hold as $\e$ goes to zero:
\begin{enumerate}
\item[$1.$] $\T_{\delta}\left(\T_{\e}^{i}(u^{\e,\delta}_{i})\right)\rightharpoonup u_i \text{ weakly in } L^{p}\left( 0,T; L^{p}\left(\Omega \times Y_i\times Z_{c}\right)\right),$
\item[$2.$] $\T_{\delta}\left(\T_{\e}^{i}(\nabla u^{\e,\delta}_{i})\right) \rightharpoonup \nabla u_i+\nabla_y \widehat{u}_i+\nabla_z \widetilde{u}_i \text{ weakly in } L^{p}(\Omega_T\times Y_i\times Z_{c}),$
\end{enumerate}
with the space $W_{\#}^{1,p}$ is given by the expression \eqref{W}.
\label{T_d_e_weak}
\end{thm}
\subsection{The Boundary Unfolding Operator}

Note that the meso-microscopic bidomain model is a dynamical boundary system at the interface of the intracellular and extracellular regions. We need to define here the unfolding operator on the boundary $\Gamma_\e,$ which developed in \cite{doinaunf06,doinaunf12,isabell}. To do that, we suppose that $\Gamma^{y}$ has a Lipschitz boundary. 
\begin{defi}
For any function $\varphi$ Lebesgue-measurable on $(0,T)\times \Gamma_\e,$ the boundary unfolding operator $\T_{\e}^{b}$ is defined as follows:
\begin{equation}
\T_{\e}^{b}(\varphi)(t,x,y)=
\begin{cases}
\varphi\left(t, \e\left[ \dfrac{x}{\e} \right]_Y +\e y \right) &\text{ a.e. for } (t, x, y) \in \widehat{\Omega}^{\e}_{T}\times \Gamma^{y}, \\ 0  &\text{ a.e. for } (t, x, y) \in \Lambda^{\e}_{T}\times \Gamma^{y}.
\end{cases}
\label{op_b_e}
\end{equation}
\end{defi}
We also list some properties of the boundary unfolding operator as given in \cite{doinaunf12}.
\begin{prop}
The boundary unfolding operator $\T_{\e}^{b}$ has the following properties:
\begin{enumerate}
\item \label{Prop1_buo} $\T_{\e}^{b}$ is linear operator from $L^p(\Gamma_{\e,T})$ to $L^p(\Omega_T\times \Gamma^{y}),$
\\
\item \label{Prop2_buo} $\T_{\e}^{b}(\varphi\psi)=\T_{\e}^{b}(\varphi)\T_{\e}^{b}(\psi), \ \forall \varphi, \psi \in L^p(\Gamma_{\e,T}), \ p \in (1,+\infty), $
\\
\item \label{Prop3_buo} For every $\varphi \in L^1(\Gamma_{\e,T}),$ we have the following integration formula: $$\dfrac{1}{\e\abs{Y}}\displaystyle\iint_{\Omega\times \Gamma^{y}} \T_{\e}^{b}(\varphi)(t,x,y) \ dx d\sigma_y=\int_{\widehat{\Gamma}_\e} \varphi(t, x) \ d\sigma_x,$$ 
\item \label{Prop4_buo} For every $\varphi \in L^p(\Gamma_{\e,T})$ with $ p \in (1,+\infty),$ one has:
$$\norm{\T_{\e}^{b}(\varphi)}_{L^{p}(\Omega\times \Gamma^{y})}=\e^{1/p}\abs{Y}^{1/p}\norm{\varphi}_{L^{p}(\widehat{\Gamma}_\e)}\leq \e^{1/p}\abs{Y}^{1/p}\norm{\varphi}_{L^{p}(\Gamma_\e)},$$

\item\label{Prop5_buo} For every $\varphi \in D(\Omega_T\times \Gamma^{y})$ and $\psi \in W^{1,1}(0,T, L^1(\Gamma_{\e})),$ the following integration by parts holds:
$$ \displaystyle\int_{0}^{T}\iint_{\Omega\times \Gamma^{y}} \T_{\e}^{b}(\pt_t \psi)\T_{\e}^{b}(\varphi) \ dx d\sigma_ydt=-\displaystyle\int_{0}^{T}\iint_{\Omega\times \Gamma^{y}} \T_{\e}^{b}( \psi)\T_{\e}^{b}(\pt_t\varphi) \ dx d\sigma_ydt. $$
\end{enumerate}
\label{prop_buo}
\end{prop} 
\begin{rem}Note that the last property (which not listed in \cite{doinaunf06,doinaunf12}) is a direct consequence of the integration by parts formula and the integration formula in property \eqref{Prop3_buo} of Proposition \ref{prop_buo}.
\end{rem}

\begin{rem} If $u_j \in L^p\left(0,T; W^{1,p}(\Omega_{j}^{\e})\right)$ for $p\in (1,+\infty),$ $\T_{\e}^{b}(u_j)$ is the trace on $\Gamma^{y}$ of $ \T_{\e}^{j}(u_j)$ with $j=i,e.$ In particular, by the standard trace theorem in $Y_j,$ there is a constant $C$ such that
\begin{equation*}
 \norm{\T_{\e}^{b}(u_j)}_{L^{p}\left(\Omega_T \times \Gamma^{y}\right)}^p\leq C\left( \norm{\T_{\e}^{j}(u_j)}_{L^{p}\left(\Omega_T \times Y_j\right)}^p+\norm{\nabla_y \T_{\e}^{j}(u_j)}_{L^{p}\left(\Omega_T \times Y_j\right)}^p\right).
\end{equation*}
From the properties of $ \T_{\e}^{j}(\cdot)$ in Proposition \ref{prop_uo}, it follows that
\begin{equation*}
 \norm{\T_{\e}^{b}(u_j)}_{L^{p}\left(\Omega_T \times \Gamma^{y}\right)}^p\leq C\left( \norm{u_j}_{L^{p}\left(\Omega_{j,T}^{\e}\right)}^p+\e^p\norm{\nabla u_j}_{L^{p}\left(\Omega_{j,T}^{\e}\right)}^p\right).
\end{equation*}
This inequality can be found as Remark 4.2 in \cite{doinaunf12}.
\label{trace_ineq}
\end{rem}
 The next result is the equivalent of Proposition \ref{T_e_strong} and Theorem \ref{T_e_weak}, to the case of functions defined on the boundary $\Gamma_\e$.
\begin{prop}Let $p\in [1,+\infty).$
\begin{enumerate}
\item[$1.$] Let $\varphi\in L^{p}(0,T;W^{1,p}(\Omega)).$ Then, one has the following convergence: 
$$\T_{\e}^{b}(\varphi) \underset{\e\rightarrow 0}{\longrightarrow} \varphi \text{ strongly in } L^{p}(\Omega_T\times \Gamma^{y}).$$
\item[$2.$] Let $\{ \varphi^\e \}$ be a sequence in $L^{p}(0,T;W^{1,p}(\Omega))$ such that :
$$ \varphi^\e \underset{\e\rightarrow 0}{\longrightarrow} \varphi \text{ strongly in } L^{p}(0,T;W^{1,p}(\Omega)).$$
Then,
$$\T_{\e}^{b}(\varphi^\e) \underset{\e\rightarrow 0}{\longrightarrow} \varphi \text{ strongly in } L^{p}(\Omega_T\times \Gamma^{y}).$$
\end{enumerate}
\label{T^b_e_strong}
\end{prop}

\begin{thm}
Let $p\in (1,+\infty).$ Suppose that $v_\e \in L^p(\Gamma_{\e,T})$ satisfies
$$ \e^{1/p}\norm{v_\e}_{L^p(\Gamma_{\e,T})}\leq C. $$
Then, there exist $v \in L^{p}\left( \Omega_{T}\times \Gamma^{y}\right)$ such that, up to a subsequence, the following convergence hold when $\e\rightarrow 0:$
\begin{equation*}
\T_{\e}^{b}(v_\e)\rightharpoonup v \text{ weakly in }  L^{p}\left(\Omega_{T}\times \Gamma^{y}\right).
\end{equation*}
\label{T^b_e_weak}
\end{thm}

\begin{proof}
The proof of this theorem can be found as Proposition 4.6 in \cite{bourgeat} and Lemma 22 in \cite{mariahom} for other details.
\end{proof}

\begin{rem}
Fix $j\in \lbrace i,e\rbrace.$ Suppose that $u^\e_{j} \in L^{2}\left( 0,T; H^{1}\left( \Omega_{j}^{\e}\right)\right)$ satisfies $ \norm{u_j^\e}_{L^{p}\left( 0,T; H^{1}\left( \Omega_{j}^{\e}\right)\right)}\leq C.$ Let $$g^\e_{j}:= u^\e_{j}\vert_{\Gamma_{\e}} \in L^2(\Gamma_{\e,T}),$$
be the trace of $u^\e_{j}$ on $\Gamma_{\e}.$ Then, there exist $u_j \in L^{2}\left( 0,T; H^{1}(\Omega)\right)$ (cf. Theorem \ref{T_e_weak}) such that, up to a subsequence, the following hold when $\e\rightarrow 0:$
\begin{equation*}
 \T_{\e}^{b}(g^\e_{j})\rightharpoonup u_{j} \text{ weakly in } L^{2}\left(\Omega_T \times \Gamma^{y}\right).
\end{equation*}
We can prove this remark by following Remark \ref{conv_trace_ini_norm}-\ref{trace_conv_proof}.
\label{trace_conv}
\end{rem}

\section{Unfolding Homogenization Method}\label{unf}
In this section, we will introduce a homogenization method based on  the unfolding operator for perforated domains and on the boundary unfolding operator. The aim is to show how to obtain the macroscopic model from the meso-microscopic bidomain model. First, the weak formulation of the meso-microscopic problem is written by another one, called "unfolded" formulation, based on unfolding operators. Then, we can pass to the limit as $\e\rightarrow 0$ in the unfolded formulation using some a priori estimates and compactness argument to obtain finally the macroscopic bidomain model.
 
 
\subsection{Intracellular problem}\label{unfintra}  Our derivation bidomain model is based on a new three-scale approach. We apply the composition of unfolding operator $\T_{\delta}(\T_{\e}^{i}(\cdot))$ in the intracellular problem to obtain its homogenized equation. Recall that $u^{\e,\delta}_i$ the solution of the following initial intracellular problem:  
 \begin{equation}
 \begin{aligned}
-\nabla \cdot (\mathrm{M}_{i}^{\e,\delta}\nabla u_{i}^{\e,\delta}) &=0 &\ \text{ in } \Omega_{i,T}^{\e,\delta}, \\ -\mathrm{M}_{i}^{\e,\delta}\nabla u_{i}^{\e,\delta} \cdot n_i=\e\left( \pt_{t} v_\e+\I_{ion}( v_\e, {w}_\e)-\I_{app,\e}\right) &=\I_m &\ \text{on} \ \Gamma_{\e,T},
\\ \pt_{t} w_\e-H(v_\e,w_\e) &=0 & \text{ on }  \Gamma_{\e,T},
\\ -\mathrm{M}_{i}^{\e,\delta}\nabla u_{i}^{\e,\delta} \cdot n_{z} & =0 &\ \text{on} \ \Gamma_{\delta,T}, 
 \end{aligned}
 \label{pbinintra}
 \end{equation}
  where the intracellular conductivity matrices $\mathrm{M}^{\e,\delta}_{i}=(\mathrm{m}^{pq}_{i})_{1 \leq p,q \leq d}$ defined by: $$\mathrm{M}^{\e,\delta}_{i}(x)=\mathrm{M}_{i}\left(\dfrac{x}{\e},\dfrac{x}{\e\delta}\right), \
  a.e. \ \text{on} \ \R^d,$$
satisfying the elliptic and periodic conditions \eqref{A_M_ie}.\\

 The problem \eqref{pbinintra} satisfies the weak formulation \eqref{Fv_i_ini}.
Since $\I_{ion}(v_\e,w_\e)=\mathrm{I}_{1,ion}(v_\e)+\mathrm{I}_{2,ion}(w_\e),$ we can rewrite the formulation \eqref{Fv_i_ini} as follows:
\begin{equation}
\begin{aligned}
\iint_{\Gamma_{\e,T}} &\e\pt_t v_\e\varphi_i \ d\sigma_xdt +\iint_{\Omega_{i,T}^{\e,\delta}}\mathrm{M}_{i}^{\e,\delta}\nabla u_{i}^{\e,\delta}\cdot\nabla\varphi_i \ dxdt
\\& +\iint_{\Gamma_{\e,T}} \e\mathrm{I}_{1,ion}(v_\e)\varphi_i \ d\sigma_xdt +\iint_{\Gamma_{\e,T}} \e\mathrm{I}_{2,ion}(w_\e)\varphi_i \ d\sigma_xdt
\\ & =\iint_{\Gamma_{\e,T}} \e\I_{app,\e}\varphi_i \ d\sigma_xdt.
\end{aligned}
\label{Fv_i}
\end{equation}
 
 We denote by $E_i$ with $i = 1,\dots, 5$ the terms of the previous equation which is rewritten as follows (to respect the order):
 \begin{equation*}
E_1+E_2+E_3+E_4=E_5.
 \end{equation*}
\subsubsection{\textbf{"Unfolded" formulation of the intracellular problem}}
The unfolding operator is used below to unfold the oscillating functions such that they are expressed in terms of global and local variables describing positions at the upper and lower heterogeneity scales, respectively. Using the properties of the unfolding operator, we rewrite the weak formulation \eqref{Fv_i} in the "unfolded" form.\\
Using property \eqref{Prop3_buo} of Proposition \ref{prop_buo}, then the first term is rewritten as follows:
\begin{align*}
E_1 
&=\iint_{\widehat{\Gamma}_{\e,T}} \e\pt_t v_\e\varphi_i \ d\sigma_xdt+\iint_{\Gamma_{\e,T}\cap\Lambda_{T}^{\e}} \e\pt_t v_\e\varphi_i \ d\sigma_xdt
\\&=\dfrac{1}{\abs{Y}}\iiint_{\Omega_{T}\times \Gamma^{y}}\T_{\e}^{b}(\pt_t v_\e)\T_{\e}^{b}(\varphi_i)\ dxd\sigma_ydt+\iint_{\Gamma_{\e,T}\cap\Lambda_{T}^{\e}} \e\pt_t v_\e\varphi_i \ d\sigma_xdt
\\& :=J_1+R_1.
\end{align*}

Similarly, we rewrite the second term using property \eqref{Prop2_cuo} of Proposition \ref{prop_cuo}:
\begin{align*}
E_2 &=\dfrac{1}{\abs{Y}}\dfrac{1}{\abs{Z}}\iiiint_{\Omega_{T}\times Y_i \times Z_{c}}\T_{\delta}(\T_{\e}^{i}(\mathrm{M}_{i}^{\e,\delta})) \T_{\delta}(\T_{\e}^{i}(\nabla u_{i}^{\e,\delta}))\T_{\delta}(\T_{\e}^{i}(\nabla\varphi_i)) \ dxdydzdt
\\& \quad +\iint_{\Lambda_{i,T}^{\e,\delta}}\mathrm{M}_{i}^{\e,\delta}\nabla u_{i}^{\e,\delta}\cdot\nabla\varphi_i \ dxdt
\\&:=J_2+R_2
\end{align*}

Due to the form of $\mathrm{I}_{k,ion},$ we use property \eqref{Prop2_buo}-\eqref{Prop3_buo} of Proposition \ref{prop_buo} to obtain $\T_{\e}^{b}\left(\mathrm{I}_{k,ion}(\cdot)\right)=\mathrm{I}_{k,ion}\left( \T_{\e}^{b}(\cdot)\right)$ for $k=1,2$ and we arrive to: 
\begin{align*}
E_3
&=\dfrac{1}{\abs{Y}}\iiint_{\Omega_{T}\times \Gamma^{y}}\T_{\e}^{b}\left(\mathrm{I}_{1,ion}(v_\e)\right) \T_{\e}^{b}(\varphi_i)\ dxd\sigma_ydt+\iint_{\Gamma_{\e,T}\cap\Lambda_{T}^{\e}} \e\mathrm{I}_{1,ion}(v_\e)\varphi_i \ d\sigma_xdt
\\&=\dfrac{1}{\abs{Y}}\iiint_{\Omega_{T}\times \Gamma^{y}}\mathrm{I}_{1,ion}\left( \T_{\e}^{b}(v_\e)\right) \T_{\e}^{b}(\varphi_i)\ dxd\sigma_ydt+\iint_{\Gamma_{\e,T}\cap\Lambda_{T}^{\e}} \e\mathrm{I}_{1,ion}(v_\e)\varphi_i \ d\sigma_xdt
\\& :=J_3+R_3\\
E_4
& =\dfrac{1}{\abs{Y}}\iiint_{\Omega_{T}\times \Gamma^{y}}\T_{\e}^{b}(\mathrm{I}_{2,ion}(w_\e))\T_{\e}^{b}(\varphi_i)\ dxd\sigma_ydt+\iint_{\Gamma_{\e,T}\cap\Lambda_{T}^{\e}} \e\mathrm{I}_{2,ion}(w_\e)\varphi_i \ d\sigma_xdt
\\& =\dfrac{1}{\abs{Y}}\iiint_{\Omega_{T}\times \Gamma^{y}}\mathrm{I}_{2,ion}\left( \T_{\e}^{b}(w_\e)\right) \T_{\e}^{b}(\varphi_i)\ dxd\sigma_ydt+\iint_{\Gamma_{\e,T}\cap\Lambda_{T}^{\e}} \e\mathrm{I}_{2,ion}(w_\e)\varphi_i \ d\sigma_xdt
\\& :=J_4+R_4
\\
E_5
&=\dfrac{1}{\abs{Y}}\iiint_{\Omega_{T}\times \Gamma^{y}}\T_{\e}^{b}(\I_{app,\e})\T_{\e}^{b}(\varphi_i)\ dxd\sigma_ydt+\iint_{\Gamma_{\e,T}\cap\Lambda_{T}^{\e}} \e\I_{app,\e}\varphi_i \ d\sigma_xdt
\\& :=J_5+R_5
\end{align*} 

 Collecting the previous estimates, we readily obtain from \eqref{Fv_i} the following "unfolded" formulation:
  \begin{equation}
\begin{aligned}
& \dfrac{1}{\abs{Y}}\iiint_{\Omega_{T}\times \Gamma^{y}}\T_{\e}^{b}(\pt_t v_\e)\T_{\e}^{b}(\varphi_i)\ dxd\sigma_ydt
\\ & +\dfrac{1}{\abs{Y}}\dfrac{1}{\abs{Z}}\iiiint_{\Omega_{T}\times Y_i \times Z_{c}}\T_{\delta}(\T_{\e}^{i}(\mathrm{M}_{i}^{\e,\delta})) \T_{\delta}(\T_{\e}^{i}(\nabla u_{i}^{\e,\delta}))\T_{\delta}(\T_{\e}^{i}(\nabla\varphi_i)) \ dxdydzdt
\\&+\dfrac{1}{\abs{Y}}\iiint_{\Omega_{T}\times \Gamma^{y}}\mathrm{I}_{1,ion}\left( \T_{\e}^{b}(v_\e)\right) \T_{\e}^{b}(\varphi_i)\ dxd\sigma_ydt
\\&+\dfrac{1}{\abs{Y}}\iiint_{\Omega_{T}\times \Gamma^{y}}\mathrm{I}_{2,ion}\left( \T_{\e}^{b}(w_\e)\right) \T_{\e}^{b}(\varphi_i)\ dxd\sigma_ydt
\\&=\dfrac{1}{\abs{Y}}\iiint_{\Omega_{T}\times \Gamma^{y}}\T_{\e}^{b}(\I_{app,\e})\T_{\e}^{b}(\varphi_i)\ dxd\sigma_ydt+R_5-R_4-R_3-R_2-R_1
\end{aligned}
\label{Fv_i_r}
\end{equation}

 Similarly, the "unfolded" formulation of \eqref{Fv_d_ini} is given by:
 \begin{equation}
\begin{aligned}
&\dfrac{1}{\abs{Y}}\iiint_{\Omega_{T}\times \Gamma^{y}}\T_{\e}^{b}(\pt_t w_\e)\T_{\e}^{b}(\phi)\ dxd\sigma_ydt
\\ & -\dfrac{1}{\abs{Y}}\iiint_{\Omega_{T}\times \Gamma^{y}}H(\T_{\e}^{b}(v_\e),\T_{\e}^{b}(w_\e))\T_{\e}^{b}(\phi)\ dxd\sigma_ydt
\\& =-\e\iint_{\Gamma_{\e,T}\cap\Lambda_{T}^{\e}} \pt_t w_\e\phi \ d\sigma_xdt+\e \iint_{\Gamma_{\e,T}\cap\Lambda_{T}^{\e}}  H(v_\e, w_\e) \phi \ d\sigma_xdt
\\&:=R_6+R_7
\end{aligned}
\label{Fv_d_r}
\end{equation}

 The intracellular homogenized model has been derived using the unfolding homogenization method at two-levels. The first level homogenization concerns the asymptotic analysis $\delta\rightarrow 0$ related to the electrical activity behavior in the micro-porous structure situated in $\Omega^{\e,\delta}_i.$ At the second level homogenization, the asymptotic analysis $\e\rightarrow 0$ is related to the electrical activity behavior in the  mesoscopic structure situated in $\Omega^{\e,\delta}_i.$ Since $\delta\leq\e,$ we pass to the limit directly in the unfolded formulation when $\e\rightarrow 0.$
\subsubsection{\textbf{Convergence of the "Unfolded" formulation}}  
In this part, we establish the passage to the limit in \eqref{Fv_i_r}-\eqref{Fv_d_r}. First, we prove that:
\begin{equation*}
R_1,\cdots,R_7 \underset{\e \rightarrow 0}{\longrightarrow} 0,
\end{equation*}
by making use of estimates \eqref{E_vw}-\eqref{E_dtvw}. So, we prove that $R_2\rightarrow 0$ when $\e\rightarrow 0$ and the proof for the other terms is similar. 
 First, by  Cauchy-Schwarz inequality,
one has 
\begin{equation*}
\begin{aligned}
R_2=\iint_{\Lambda_{i,T}^{\e,\delta}}\mathrm{M}_{i}^{\e,\delta}(x)\nabla u_{i}^{\e,\delta}\cdot\nabla\varphi_i \ dxdt\leq\norm{\mathrm{M}_{i}^{\e,\delta}\nabla u_{i}^{\e,\delta}}_{L^2\left(\Omega_{i,T}^{\e,\delta}\right)} \left(\iint_{\Lambda_{i,T}^{\e,\delta}} \abs{\nabla\varphi_i}^2 \ dxdt\right)^{1/2}. 
\end{aligned}
\end{equation*}
In addition, we observe that $\abs{\Lambda_{i}^{\e,\delta}}\rightarrow 0$ and $\nabla\varphi_i \in L^2(\Omega_i^{\e,\delta}).$ Consequently, by Lebesgue dominated convergence theorem, one gets 
\begin{equation*}
\iint_{\Lambda_{i}^{\e,\delta}} \abs{\nabla\varphi_i}^2\rightarrow 0, \text{ as } \e\rightarrow 0.
\end{equation*} 
Finally, by using Holder inequality, the result follows by using estimate \eqref{E_u} and assumption \eqref{A_M_ie}. 
   
Let us now elaborate the convergence results of $J_1,\cdots,J_5$. First, we choose a special form of test functions to capture the mesoscopic and microscopic informations at each structural level. Then, we consider that the test functions have the following form:
\begin{equation}
\varphi_{i}^{\e,\delta}=\Psi_i(t,x)+\e\Psi_1(t,x)\Phi_1^{\e}(x) +\e\delta\Psi_2(t,x)\Phi_2^{\e}(x) \Theta^{\e,\delta}(x),
\end{equation}
with functions $\Phi_k^{\e}$ and $\Theta^{\e,\delta}$ defined by: 
$$\Phi_k^{\e}(x)=\Phi_k\left( \dfrac{x}{\e}\right), \ \text{ for } k=1,2 \text{ and } \Theta^{\e,\delta}(x)=\Theta\left( \dfrac{x}{\e\delta}\right),$$
where $\Psi_i, \Psi_{k},$  are in $D(\Omega_T),$ $\Phi_k$ in $H_{\#}^1(Y_i)$ for $k=1,2$ and $\Theta$ in $H_{\#}^1(Z_{c}).$
 We have:
 \begin{equation*}
 \nabla \varphi_{i}^{\e,\delta}=\nabla_x \Psi_i+\Psi_1\nabla_y\Phi_1^{\e}+\Psi_2 \Phi_2^{\e}\nabla_z\Theta^{\e,\delta}+\e\nabla_x\Psi_1\Phi_1^{\e}+\e\delta\nabla_x\Psi_2 \Phi_2^{\e} \Theta^{\e,\delta}+\delta\Psi_2\nabla_y \Phi_2^{\e} \Theta^{\e,\delta}.
 \end{equation*}
 Due to the regularity of test functions together with Proposition \ref{T_d_e_strong} and  Proposition \ref{T^b_e_strong}, there holds: 
 \begin{align*}
&\T_{\delta}\left(\T_{\e}^{i}(\varphi_{i}^{\e,\delta})\right)\rightarrow \Psi_i \text{ strongly in } L^{2}\left(\Omega_T \times Y_i\times Z_{c}\right),
\\& \T_{\delta}\left(\T_{\e}^{i}(\Psi_1\Phi_1^{\e})\right)\rightarrow \Psi_1(t,x) \Phi_1(y)\text{ strongly in } L^{2}\left(\Omega_T \times Y_i\times Z_{c}\right),
\\& \T_{\delta}\left(\T_{\e}^{i}(\Psi_2\Phi_2^{\e}\Theta^{\e,\delta})\right)\rightarrow \Psi_2(t,x) \Phi_2(y) \Theta(z)\text{ strongly in } L^{2}\left(\Omega_T \times Y_i\times Z_{c}\right),
\\& \T_{\delta}\left(\T_{\e}^{i}\left( \nabla\varphi_{i}^{\e,\delta}\right) \right)\rightarrow \nabla_x \Psi_i+\Psi_1\nabla_y\Phi_1+\Psi_2 \Phi_2\nabla_z\Theta \text{ strongly in } L^{2}\left(\Omega_T \times Y_i\times Z_{c}\right),
\\& \T_{\e}^{b}(\varphi_{i}^{\e,\delta}) \rightarrow \Psi_i \text{ strongly in } L^{2}(\Omega_T\times \Gamma^{y}).
\end{align*}

Next, we want to use the a priori estimates \eqref{E_vw}-\eqref{E_dtvw} to verify that the remaining terms of the equations are weakly convergent in the unfolded formulation \eqref{Fv_i_r}-\eqref{Fv_d_r}. Using estimation \eqref{E_u}, we deduce from Theorem \ref{T_d_e_weak} that there exist $u_i \in L^{2}\left( 0,T; H^{1}(\Omega)\right),\widehat{u}_i\in L^{2}\left( 0,T; L^{2}\left(\Omega, H_{\#}^{1}(Y_i)\right)\right)$ and $\widetilde{u}_i\in L^{2}\left( 0,T; L^{2}\left(\Omega \times Y_i, H_{\#}^{1}(Z_{c})\right)\right)$ such that, up to a subsequence, the following convergences hold as $\e$ goes to zero:
\begin{align*}
&\T_{\delta}\left(\T_{\e}^{i}(u^{\e,\delta})\right)\rightharpoonup u_i \text{ weakly in } L^{2}\left( 0,T; L^{2}\left(\Omega \times Y_i\times Z_{c}\right)\right),
\\ & \T_{\delta}\left(\T_{\e}^{i}(\nabla u^{\e,\delta})\right) \rightharpoonup \nabla u_i+\nabla_y \widehat{u}_i+\nabla_z \widetilde{u}_i \text{ weakly in } L^{2}(\Omega_T\times Y_i\times Z_{c}),
\end{align*}
with the space $H_{\#}^{1}$ is given by \eqref{W}.
Thus, since $\T_{\delta}\left(\T_{\e}^{i}\left(\mathrm{M}_{i}^{\e,\delta}\right) \right)\rightarrow \mathrm{M}_{i}$ a.e in $\Omega\times Y_i \times Z_{c},$  one obtains:
$$J_2\rightarrow\dfrac{1}{\abs{Y}}\dfrac{1}{\abs{Z}}\iiiint_{\Omega_{T}\times Y_i \times Z_{c}}\mathrm{M}_{i} \left[ \nabla u_{i}+\nabla_y \widehat{u}_i+\nabla_z \widetilde{u}_i\right] \left[ \nabla_x \Psi_i+\Psi_1\nabla_y\Phi_1+\Psi_2 \Phi_2\nabla_z\Theta\right]  \ dxdydzdt.$$

\begin{rem} Since $u_{i}$ is independent of $y$ and $z$ then it does not oscillate "rapidly". This is why now expect $u_{i}$ to be the "homogenized solution". To find the homogenized equation, it is sufficient to find an equation in $\Omega$ satisfied by $u_{i}$ independent on $y$ and $z.$
\end{rem} 
 
 Furthermore, we need to establish the weak convergence of the unfolded sequences that corresponds to $v_\e, w_\e$ and $\I_{app,\e}.$ In order to establish the convergence of $\T_{\e}^{b}(\pt_t v_\e),$ we use estimation \eqref{E_dtvw} to get
 $$ \norm{\T_{\e}^{b}(\pt_t v_\e)}_{L^{2}(\Omega_T \times \Gamma^{y})}\leq\e^{1/2}\abs{Y}^{1/2}\norm{\pt_t v_\e}_{L^{2}(\Gamma_{\e,T})}\leq C.$$
So there exists $V\in L^{2}(\Omega_T \times \Gamma^{y})$ such that $\T_{\e}^{b}(\pt_t v_\e)\rightharpoonup V$  weakly in $L^{2}(\Omega_T \times \Gamma^{y}).$ By a classical integration argument, one can show that $V=\pt_t v.$ Therefore, we deduce from Theorem \ref{T^b_e_weak} that
$$\T_{\e}^{b}(\pt_t v_\e)\rightharpoonup \pt_t v \text{ weakly in } L^{2}(\Omega_T \times \Gamma^{y}).$$
Thus, we obtain
$$J_1=\dfrac{1}{\abs{Y}}\iiint_{\Omega_{T}\times \Gamma^{y}} \T_{\e}^{b}(\pt_t v_\e) \T_{\e}^{b}(\varphi_i)\ dxd\sigma_ydt \rightarrow \dfrac{1}{\abs{Y}}\iiint_{\Omega_{T}\times \Gamma^{y}}\pt_t v \Psi_i \ dxd\sigma_ydt.$$
\begin{rem}\label{conv_trace_ini_norm}
\begin{enumerate}[label=(\alph*)]
\item \label{trace_conv_proof} We observe that the limit $v$ coincides with $u_i-u_e.$ Indeed, it follows that, by using property \eqref{Prop3_buo} of Proposition \ref{prop_buo},
\begin{align*}
\e\iint_{\Gamma_{\e,T}} v_\e\varphi \ d\sigma_xdt
&=\e\iint_{\widehat{\Gamma}_{\e,T}} v_\e\varphi \ d\sigma_xdt+\e\iint_{\Gamma_{\e,T}\cap\Lambda_{T}^{\e}} v_\e\varphi \ d\sigma_xdt
\\& =\dfrac{1}{\abs{Y}}\iiint_{\Omega_{T}\times \Gamma^{y}} \T_{\e}^{b}(v_\e) \T_{\e}^{b}(\varphi)\ dxd\sigma_ydt+\iint_{\Gamma_{\e,T}\cap\Lambda_{T}^{\e}} \e v_\e\varphi \ d\sigma_xdt
\\ & := J_{\e}+R_{\e},
\end{align*}
for all $\varphi \in C^{\infty}(\Omega_T).$ We can similarly prove that $R_{\e}\rightarrow 0$ as in the proof for the terms $R_1,\dots, R_7$ when $\e$ goes to zero. Then, it sufficient to prove the convergence results of $J_\e$ when $\e \rightarrow 0.$ On the one hand, to establish the convergence of $\T_{\e}^{b}(v_\e),$ we use estimation \eqref{E_vr} to get
 $$ \norm{\T_{\e}^{b}(v_\e)}_{L^{2}(\Omega_T \times \Gamma^{y})}\leq\e^{1/2}\abs{Y}^{1/2}\norm{v_\e}_{L^{2}(\Gamma_{\e,T})}\leq C.$$
So, we deduce from Theorem \ref{T^b_e_weak} that there exists $v\in L^{2}(\Omega_T \times \Gamma^{y})$ such that $\T_{\e}^{b}(v_\e)\rightharpoonup v$  weakly in $L^{2}(\Omega_T \times \Gamma^{y}).$ 
Therefore, we obtain
$$J_\e=\dfrac{1}{\abs{Y}}\iiint_{\Omega_{T}\times \Gamma^{y}} \T_{\e}^{b}(v_\e) \T_{\e}^{b}(\varphi)\ dxd\sigma_ydt \rightarrow \dfrac{1}{\abs{Y}}\iiint_{\Omega_{T}\times \Gamma^{y}} v \varphi \ dxd\sigma_ydt.$$ On the other hand, since $v_{\e}=\left(u_{i}^{\e}-u_{e}^{\e}\right)\vert_{\Gamma_{\e,T}}$ and, due to the fact that $\T_{\e}^{b}(u_{j}^{\e})$ is the trace on $\Gamma^{y}$ of $ \T_{\e}^{j}(u_{j}^{\e})$ for $j=i,e$ $($consult Remark \ref{trace_ineq}$)$, we can rewrite $J_{\e}$ as follows
\begin{align*}
J_{\e} &=\dfrac{1}{\abs{Y}}\iiint_{\Omega_{T}\times \Gamma^{y}} \T_{\e}^{b}(v_\e) \T_{\e}^{b}(\varphi)\ dxd\sigma_ydt
\\ & =\dfrac{1}{\abs{Y}}\iiint_{\Omega_{T}\times \Gamma^{y}} \T_{\e}^{b}\left(\left(u_{i}^{\e}-u_{e}^{\e}\right)\vert_{\Gamma_{\e,T}}\right) \T_{\e}^{b}(\varphi)\ dxd\sigma_ydt
\\ & =\dfrac{1}{\abs{Y}}\iiint_{\Omega_{T}\times \Gamma^{y}} \left(\T_{\e}^{i}\left(u_{i}^{\e}\right) -\T_{\e}^{e}\left(u_{e}^{\e}\right)\right)\vert_{\Omega_{T}\times\Gamma^{y}} \T_{\e}^{b}(\varphi)\ dxd\sigma_ydt.
\end{align*}
Now, by using Theorem \ref{T_e_weak}, there exist $u_j \in L^{2}\left( 0,T; H^1(\Omega)\right)$ such that $\T_{\e}^{j}(u^{\e}_{j})\rightharpoonup u_{j}$ weakly in $L^{2}\left( 0,T; L^{2}\left(\Omega, H^1(Y_{j})\right)\right),$ for $j=i,e.$ Thus, we deduce
\begin{align*}
J_{\e} & \rightarrow \dfrac{1}{\abs{Y}}\iiint_{\Omega_{T}\times \Gamma^{y}}\left(u_{i}-u_{e}\right)\vert_{\Omega_{T}\times\Gamma^{y}} \varphi \ dxd\sigma_ydt.
\end{align*}
Herein, we used the integration formula of the operator $\T_{\e}^{b}$ in the first step and exploited that $v$ is independent of $y$ and $v$ coincides with $u_{i}-u_{e}$ in the last step. This prove Remark \ref{trace_conv} for $v_\e=(u_i^{\e}-u_{e}^{\e})\vert_{\Gamma_{\e}}.$\\
\item \label{conv_cond_ini} Moreover, we have assumed that the initial data $v_{0,\e},w_{0,\e}$ in \eqref{cond_ini_vw}, are also uniformly bounded in the adequate norm $($see assumption \eqref{A_vw0}$)$. Therefore, in the same way as the previous proof \ref{trace_conv_proof}, using again the integration formula $(3)$ of the operator $\T_{\e}^{b}$, we know that there exist $v'_{0}, w'_{0} \in L^{2}(\Omega_T \times \Gamma^{y})$ such that, up to a subsequence, 
\begin{align*}
& \e\iint_{\Gamma_{\e}} v_{0,\e}\phi \ d\sigma_x
\rightarrow\dfrac{\abs{\Gamma^{y}}}{\abs{Y}} \int_{\Omega} v_{0}  \phi \ dx,
\\& \e\iint_{\Gamma_{\e}} w_{0,\e}\phi \ d\sigma_x
\rightarrow\dfrac{\abs{\Gamma^{y}}}{\abs{Y}} \int_{\Omega} w_{0} \phi \ dx,
\end{align*}
for all $\phi \in C^{\infty}(\Omega),$ where $v_{0}= \dfrac{1}{\abs{\Gamma^{y}}}\displaystyle\int_{\Gamma^{y}}v'_{0} \ d\sigma_y$ and $w_{0}= \dfrac{1}{\abs{\Gamma^{y}}}\displaystyle\int_{\Gamma^{y}}w'_{0} \ d\sigma_y.$
\item Finally, one can pass to the limit in the normalization condition defined by \eqref{normalization_cond} to recover a condition on the average of $u_{e}$ (the limit of $\T_{\e}^{e}(u_{e}^{\e})$) and we get the following equation, for all $\varphi \in C^{0}([0,T]),$
\begin{align*}
0=\int_{0}^{T}\left(\int_{\Omega_{e}^{\e}} u_{e}^{\e}dx\right) \varphi \ dt
&=\dfrac{1}{\abs{Y}}\int_{0}^{T}\left(\iint_{\Omega\times Y_{e}} \T_{\e}^{e}(u_{e}^{\e})dxdy\right)  \varphi\ dt+ \int_{0}^{T}\left( \int_{\Lambda_{e}^{\e}}  u_{e}^{\e} dx\right)  \varphi\ dt
\\ & \rightarrow 0= \dfrac{\abs{Y_e}}{\abs{Y}}\int_{0}^{T}\left(\int_{\Omega} u_{e}dx\right)  \varphi\ dt,
\end{align*}
where the second term in the previous equality goes to zero as the proof for the terms $R_1,\dots, R_7$ when $\e\rightarrow 0$. This implies that we have, for almost all $t\in [0,T],$
\begin{equation*}
\int_{\Omega} u_{e} (t,x)dx=0.
\end{equation*}
\end{enumerate}  
\end{rem}
Now, making use of estimate \eqref{E_vw} with property \eqref{Prop4_buo} of Proposition \ref{prop_buo}, one has
 $$ \norm{\T_{\e}^{b}(w_\e)}_{L^{2}(\Omega_T \times \Gamma^{y})}\leq\e^{1/2}\abs{Y}^{1/2}\norm{w_\e}_{L^{2}(\Gamma_{\e,T})}\leq C.$$
Then, up to a subsequence, 
$$\T_{\e}^{b}(w_\e)\rightharpoonup w \text{ weakly in } L^{2}(\Omega_T \times \Gamma^{y}).$$
So, by linearity of $\mathrm{I}_{2,ion},$ we have:
$$J_4=\dfrac{1}{\abs{Y}}\iiint_{\Omega_{T}\times \Gamma^{y}}\mathrm{I}_{2,ion}\left( \T_{\e}^{b}(w_\e)\right) \T_{\e}^{b}(\varphi_i)\ dxd\sigma_ydt\rightarrow \dfrac{1}{\abs{Y}}\iiint_{\Omega_{T}\times \Gamma^{y}}\mathrm{I}_{2,ion}(w) \Psi_i \ dxd\sigma_ydt.$$
Similarly, we can prove the convergence of $\T_{\e}^{b}(\I_{app,\e}),$ by using assumption \eqref{A_iapp}, to get
 $$ \norm{\T_{\e}^{b}(\I_{app,\e})}_{L^{2}(\Omega_T \times \Gamma^{y})}\leq\e^{1/2}\abs{Y}^{1/2}\norm{\I_{app,\e}}_{L^{2}(\Gamma_{\e,T})}\leq C.$$
So we can conclude from Theorem \ref{T^b_e_weak} that there exists $\I_{app,0}\in L^{2}(\Omega_T \times \Gamma^{y})$ such that $\T_{\e}^{b}(\I_{app,\e})\rightharpoonup \I_{app,0}$  weakly in $L^{2}(\Omega_T \times \Gamma^{y}).$ 
Thus, we obtain the following  convergence:
$$J_5=\dfrac{1}{\abs{Y}}\iiint_{\Omega_{T}\times \Gamma^{y}}\T_{\e}^{b}(\I_{app,\e})\T_{\e}^{b}(\varphi_i)\ dxd\sigma_ydt\rightarrow \dfrac{\abs{\Gamma^{y}}}{\abs{Y}}\iint_{\Omega_{T}}\I_{app}\Psi_i\ dxdt,$$
where $\I_{app}= \dfrac{1}{\abs{\Gamma^{y}}}\displaystyle\int_{\Gamma^{y}}\I_{app,0} \ d\sigma_y.$
 
 It remains to obtain the limit of $J_3$ containing the ionic function $\mathrm{I}_{1,ion}.$ By the regularity of $\varphi_{i}$, it sufficient to show the weak convergence of $\mathrm{I}_{1,ion}\left( \T_{\e}^{b}(v_\e)\right) $ to $\mathrm{I}_{1,ion}(v)$ in $L^2(\Omega_T\times
\Gamma^{y}).$  Due to the non-linearity of $\mathrm{I}_{1,ion},$ the weak convergence will not be enough. Therefore, we need also the strong convergence of $\T_{\e}^{b}(v_\e)$ to $v$ in $L^2(\Omega_T\times
\Gamma^{y})$ by using Kolmogorov-Riesz type compactness criterion \ref{kolmo}. 
Next, we prove by Vitali's Theorem the strong convergence of $\mathrm{I}_{1,ion}\left( \T_{\e}^{b}(v_\e)\right)$ to $\mathrm{I}_{1,ion}(v)$ in $L^q(\Omega_T\times \Gamma^{y}), \ \forall q\in[1,r/(r-1))$ with $r\in (2,+\infty).$\\
To cope with this, we derive the convergence of the nonlinear term $\mathrm{I}_{1,ion},$ in the following lemma:
\begin{lem}
The following convergence holds:\begin{equation*}
  \T_{\e}^{b}(v_\e)\rightarrow v \text{ strongly in } L^2(\Omega_T\times \Gamma^{y}),
\end{equation*}
 as $\e\rightarrow 0.$ Moreover, we have:
\begin{equation*}
\mathrm{I}_{1,ion}\left( \T_{\e}^{b}(v_\e)\right)\rightarrow \mathrm{I}_{1,ion}(v)\text{ strongly in } L^q(\Omega_T\times \Gamma^{y}), \ \forall q\in[1,r/(r-1)),
\end{equation*}
as $\e\rightarrow 0.$
\end{lem}
\begin{proof}
We follow the same idea to the proof of Lemma 5.3 in \cite{bendunf19}.
The proof of the first convergence is based on the Kolmogorov compactness criterion, which is recalled for the convenience of the reader in Proposition \ref{kolmo}. It is carried out in three conditions:

 $\textbf{(i)}$ Let $A\subset\Omega$ a measurable set. We define the sequence $\lbrace v_{A}^{\e} \rbrace_{\e>0}$ as follows:
\begin{equation*}
v_{A}^{\e}(t,y):=\int_{A} \T_{\e}^{b}(v_\e)(t,x,y) \ dx, \text{ for a.e. } t\in (0,T), \ y\in \Gamma^{y}.
\end{equation*} 
 It remains to show that the sequence $v_{A}^{\e} \in L^2\left( 0,T;H^{1/2}(\Gamma^{y})\right)$ is relatively compact in the space $L^2\left( 0,T;L^{2}(\Gamma^{y})\right).$ Since the embedding $H^{1/2}(\Gamma^{y}) \hookrightarrow L^{2}(\Gamma^{y})$ is compact, we have to show that the sequence $v_{A}^{\e}$ is bounded in $L^2\left( 0,T;H^{1/2}(\Gamma^{y})\right) \cap H^1\left( 0,T;L^{2}(\Gamma^{y})\right).$

 We first observe that
 \begin{equation*}
 \begin{aligned}
 \norm{v_{A}^{\e}}_{H^{1/2}(\Gamma^{y})}^2 
 & = \displaystyle \int_{\Gamma^{y}}\card{\int_{A} \T_{\e}^{b}(v_\e)(t,x,y) \ dx}^2 d\sigma_y
 \\ & \quad +\displaystyle \iint_{\Gamma^{y}\times \Gamma^{y}} \int_{A} \dfrac{\card{ \T_{\e}^{b}(v_\e)(t,x,y_1)-\T_{\e}^{b}(v_\e)(t,x,y_2)}^2}{\card{ y_1-y_2 }^{d+1}} \ dxd\sigma_{y_1}d\sigma_{y_2}
 \\& :=\norm{v_{A}^{\e}}_{L^{2}(\Gamma^{y})}^2+ \norm{v_{A}^{\e}}_{H^{1/2}_0(\Gamma^{y})}^2.
 \end{aligned}
 \end{equation*}
 
 With Fubini and Cauchy-Schwarz inequality and the a priori estimate \eqref{E_vw}, one has
 \begin{align*}
 \norm{v_{A}^{\e}}_{L^{2}(\Gamma^{y}_T)}^2 
 &\leq C \displaystyle \int_{0}^{T}\int_{\Omega}\int_{\Gamma^{y}}\card{ \T_{\e}^{b}(v_\e)(t,x,y) }^2 d\sigma_ydxdt
 \\& \leq C \norm{\sqrt{\e}v_{\e}}_{L^{2}(\Gamma_{\e,T})}^2 \leq C.
 \end{align*}
 Next, we need to bound the $H^{1/2}_0$ semi-norm. Since $v_\e=\left( u_{i}^{\e}-u_{e}^{\e}\right){\vert \Gamma_{\e}},$ we use again Fubini and Jensen inequality together with the trace inequality in Remark \ref{trace_ineq} to obtain
 \begin{align*}
\norm{v_{A}^{\e}}_{H^{1/2}_0(\Gamma^{y})}^2
&\leq C\left[\int_{\Omega}\norm{\T_{\e}^{b}(v_\e)}_{H^{1/2}_0(\Gamma^{y})}^2dxdt\right] 
\\ & \leq C\left[ \norm{u_{i}^{\e,\delta}}_{L^{2}(\Omega_{i}^{\e,\delta})}^2+\e^2\norm{\nabla u_{i}^{\e,\delta}}_{L^{2}(\Omega_{i}^{\e,\delta})}^2+\norm{u_{e}^{\e}}_{L^{2}(\Omega_{e}^{\e})}^2+\e^2\norm{\nabla u_{e}^{\e}}_{L^{2}(\Omega_{e}^{\e})}^2\right].
 \end{align*}
 Hence, integrating over $(0,T)$ and using the a priori estimates \eqref{E_u}, we have showed that the sequence $v_{A}^{\e}$ is bounded in $L^2\left( 0,T;H^{1/2}(\Gamma^{y})\right).$
 
 By a similar argument and making use of the estimate \eqref{E_dtvw} on $\e^{1/2}\pt_t v_\e$, we can also show that
\begin{equation*}
\norm{\pt_t v_{A}^{\e}}_{L^{2}(\Gamma^{y}_T)} \leq C.
\end{equation*}
Finally, we deduce that the sequence $v_{A}^{\e}$ is bounded in $L^2\left( 0,T;H^{1/2}(\Gamma^{y})\right) \cap H^1\left( 0,T;L^{2}(\Gamma^{y})\right)$ and due to the Aubin-Lions Lemma the sequence is relatively compact in $L^2\left( 0,T;L^{2}(\Gamma^{y})\right).$
 
 $\textbf{(ii)}$ Due to the decomposition of the domain in Definition \ref{domain_unf}, $\Omega$ can always be represented by a union of scaled and translated reference cells. Fix $\e>0$ and let $k \in \Xi_\e,$ be an index set such that
\begin{equation*}
\widehat{\Omega}^{\e}= \underset{k \in \Xi_\e }{\bigcup} \e (k_{\ell} + Y), \text{ with } k_\ell:=( k_1\ell^\text{mes}_1,\dots,  k_d \ell^\text{mes}_d ).
\end{equation*}
Note that $x \in \e (k_{\ell} +Y) \Leftrightarrow \left[ \dfrac{x}{\e} \right]_Y=k_{\ell}.$ For every fixed $k \in \Xi_\e,$ we subdivide the cell $\e (k_{\ell} +Y)$ into subsets $\e \left(k_{\ell} +Y\right)^\sigma$ with $\sigma\in\left\lbrace 0,1 \right\rbrace^d,$ defined as follows
\begin{equation*}
\e (k_{\ell} +Y)^\sigma :=\left\lbrace x \in \e (k_{\ell} +Y)  : \e\left[\dfrac{x+\e\left\lbrace  \dfrac{h}{\e} \right\rbrace_Y}{\e}\right]_Y =\e (k_{\ell} +\sigma) \right\rbrace,
\end{equation*}
for a given $h \in \R^d.$ It holds $\e (k_{\ell} +Y)=\underset{\sigma\in\left\lbrace 0,1 \right\rbrace^d }{\bigcup}\e (k_{\ell} +Y)^\sigma.$

 We use the same notation as in Proposition \ref{kolmo}. Now, we compute
\begin{align*}
\norm{\tau_h\T_{\e}^{b}(v_\e)-\T_{\e}^{b}(v_\e)}_{L^{2}\left((0,T)\times \Omega^{h}_{\lambda} \times \Gamma^{y}\right)}^2
&=\norm{\tau_h\T_{\e}^{b}(v_\e)-\T_{\e}^{b}(v_\e)}_{L^{2}\left((0,T)\times (\Omega^{h}_{\lambda}\cap \widehat{\Omega}^{\e}) \times \Gamma^{y}\right)}^2
\\& \quad +\norm{\tau_h\T_{\e}^{b}(v_\e)-\T_{\e}^{b}(v_\e)}_{L^{2}\left((0,T)\times (\Omega^{h}_{\lambda}\setminus\widehat{\Omega}^{\e}) \times \Gamma^{y}\right)}^2
\\& := E_{1,\e}^{h}+E_{2,\e}^{h}.
\end{align*}
Proceeding in a similar way to \cite{soren,maria07}, we first estimate $E_{1,\e}^{h}$ using the above decomposition of the domain as follows:
\begin{align*}
E_{1,\e}^{h}& =\sum\limits_{k \in \Xi_\e}\int_{0}^{T} \int_{\e (k_{\ell} +Y)} \int_{\Gamma^{y}} \card{v_\e\left(t,\e\left[\dfrac{x+h}{\e} \right]_Y+\e y  \right)-v_\e\left(t,\e\left[\dfrac{x}{\e} \right]_Y+\e y  \right) }^2 d\sigma_ydxdt
\\ & =\sum\limits_{k \in \Xi_\e}\sum\limits_{\sigma\in\left\lbrace 0,1 \right\rbrace^d }\int_{0}^{T} \int_{\e (k_{\ell} +Y)^\sigma} \int_{\Gamma^{y}} \card{v_\e\left(t,\e\left( k_{\ell}+\sigma+\left[\dfrac{h}{\e} \right]_Y\right)+\e y  \right)-v_\e\left(t,\e k_{\ell}+\e y  \right) }^2 d\sigma_ydxdt
\\ & \leq\sum\limits_{k \in \Xi_\e}\sum\limits_{\sigma\in\left\lbrace 0,1 \right\rbrace^d }\int_{0}^{T} \int_{\e (k_{\ell} +Y)} \int_{\Gamma^{y}} \card{v_\e\left(t,\e\left( k_{\ell}+\sigma+\left[\dfrac{h}{\e} \right]_Y\right)+\e y  \right)-v_\e\left(t,\e k_{\ell}+\e y  \right) }^2 d\sigma_ydxdt
\\ & \leq \sum\limits_{\sigma\in\left\lbrace 0,1 \right\rbrace^d }\int_{0}^{T} \int_{\widehat{\Omega}^{\e}} \int_{\Gamma^{y}} \card{\T_{\e}^{b}v_\e\left(t,x+\e\left( \sigma+\left[\dfrac{h}{\e} \right]_Y\right), y  \right)-\T_{\e}^{b}v_\e\left(t, x, y \right) }^2 d\sigma_ydxdt,
\end{align*} 
which by using the integration formula \eqref{Prop4_buo} $($for $p=2)$ of Proposition \ref{prop_buo} is equal to
\begin{equation*}
\sum\limits_{\sigma\in\left\lbrace 0,1 \right\rbrace^d }\e  \abs{Y}\int_{0}^{T}  \int_{\Gamma_{\e}} \card{v_\e\left(t,x+\e\left( \sigma+\left[\dfrac{h}{\e} \right]_Y\right)  \right)-v_\e\left(t, x \right) }^2 d\sigma_ydt.
\end{equation*}

For a given small $\gamma>0,$ we can choose an $\e$ small enough such that $\card{\e \sigma+\e\left[\dfrac{h}{\e} \right]_Y}< \gamma.$ This amounts to saying that in order to estimate $E_{1,\e}^{h},$ it is sufficient to obtain estimates for given $\ell\in \mathbb{Z}^d,$ $\abs{\e \ell}<\gamma$ of 
\begin{equation}
\norm{v_\e\left(t,x+\e\ell\right)-v_\e\left(t, x \right)}_{L^{2}\left((0,T) \times \Gamma_{\e,K}\right)}^2,
\label{E_trans_v}
\end{equation}
where $\Gamma_{\e,K}=\Gamma_{\e}\cap K$ with $K\subset \Omega$ an open set. 
 
 In order to estimate the norm \eqref{E_trans_v}, we test the variational equation \eqref{Fv_i_ini} for $\tau_{\e\ell}u_{e}^{\e}-u_{e}^{\e} $ with $\varphi_{i}=\eta^2\left( \tau_{\e\ell}u_{i}^{\e,\delta}-u_{i}^{\e,\delta}\right) $ and $\varphi_{e}=\eta^2\left( \tau_{\e\ell}u_{e}^{\e}-u_{e}^{\e}\right),$ where $\eta \in D(K)$ is a cut-off function with $0\leq \eta \leq 1,$ $\eta=1 $ in $K$ and zero outside a small neighborhood $K'$ of $K.$ Proceeding exactly as Lemma 5.2 in \cite{bendunf19}, Gronwall's inequality and the assumptions on the initial data give the following result:
 $$\e\norm{v_\e\left(t,x+\e\ell\right)-v_\e\left(t, x \right)}_{L^{2}\left((0,T) \times \Gamma_{\e,K}\right)}^2\leq C\e\abs{\ell},$$
 where $C$ is a positive constant.
Then, we obtain by using the previous estimate
\begin{equation}
E_{1,\e}^{h}\leq C\left(\abs{h}+\e\right).
\label{E_1e_h}
\end{equation}
Hence, we can deduce that $E_{1,\e}^{h}\rightarrow 0$ as $h\rightarrow 0$ uniformly in $\e$, as in \cite{mariahom}. Indeed, to prove that
\begin{equation}
\forall \rho>0, \exists \mu>0 \text{ such that } \forall \e>0, \ \forall h, \ \abs{h}\leq \mu \Rightarrow E_{1,\e}^{h}<\rho,
\label{E_1}
\end{equation}
one identifies two cases:
\begin{itemize}
\item[$(a)$] For $0<\e<\dfrac{\rho}{2C}:$ take $\mu=\dfrac{\rho}{2C},$ then, from \eqref{E_1e_h}, we get that condition \eqref{E_1} holds for $\abs{h}\leq \mu.$ 
\item[$(b)$] For $\dfrac{\rho}{2C}<\e < 1:$ we consider sequences $\e$ of the form $\e_k=\dfrac{1}{k},$ $k\in \mathbb{N},$ there are finitely many elements $\e_k$ in the interval $(\frac{\rho}{2C},1)$ and for each $\e_k,$ $\exists \mu_k=\mu(\e_k)$ such that $\forall h, \ \abs{h}\leq \mu_k,$ condition \eqref{E_1} holds, due to the continuity of translations in the mean of $L^2$-functions. Thus choosing $\mu=\min\lbrace\frac{\rho}{2C}, \mu_k\rbrace,$ property \eqref{E_1} is proved.
\end{itemize}

It easy to check that $$E_{2,\e}^{h}=\norm{\tau_h\T_{\e}^{b}(v_\e)}_{L^{2}\left((0,T)\times (\Omega^{h}_{\lambda}\setminus\widehat{\Omega}^{\e}) \times \Gamma^{y}\right)}^2\leq \norm{\tau_h\T_{\e}^{b}(v_\e)}_{L^{2}\left((0,T)\times (\Omega_{\lambda}\setminus\widehat{\Omega}^{\e}) \times \Gamma^{y}\right)}^2.$$
Hence, we can deduce that $E_{2,\e}^{h}\rightarrow 0$ as $h\rightarrow 0$ uniformly in $\e.$ Indeed, to prove that
\begin{equation}
\forall \rho>0, \exists \mu>0 \text{ such that } \forall \e>0, \ \forall h, \ \abs{h}\leq \mu \Rightarrow E_{2,\e}^{h}<\rho,
\label{E_2}
\end{equation}
one identifies two cases:
\begin{itemize}
\item[$(a)$] For $\e$ small enough, say $\e<\e_0,$ $\Omega_{\lambda}\subset\widehat{\Omega}^{\e},$ then $E_{2,\e}^{h}=0.$
\item[$(b)$] For $\e_0<\e < 1:$ we consider sequences $\e$ of the form $\e_k=\dfrac{1}{k},$ $k\in \mathbb{N},$ there are finitely many elements $\e_k$ in the interval $(\e_0,1)$ and for each $\e_k,$ $\exists \mu_k=\mu(\e_k)$ such that $\forall h, \ \abs{h}\leq \mu_k,$ condition \eqref{E_2} holds, due to the continuity of translations in the mean of $L^2$-functions. Thus choosing $\mu=\min\lbrace \mu_k\rbrace,$ property \eqref{E_2} is proved.
\end{itemize}
This ends the proof of the condition (ii) in Proposition \ref{kolmo}.

$\textbf{(iii)}$ The last condition follows from the a priori estimate \eqref{E_vr}. Indeed, we have:
\begin{equation*}
\int_{0}^{T}\int_{\Omega\setminus\Omega_{\lambda}}\card{\T_{\e}^{b}(v_\e)}^2 dxdt \leq \abs{\Omega\setminus\Omega_{\lambda}}^{\frac{r-2}{r}}\left(\int_{\Omega_T}\card{\T_{\e}^{b}(v_\e)}^r dxdt \right)^{\frac{2}{r}}\leq C \abs{\Omega\setminus\Omega_{\lambda}}^{\frac{r-2}{r}}.
\end{equation*}
The conditions (i)-(iii) imply that the Kolmogorov criterion for $\T_{\e}^{b}(v_\e)$ holds true in $L^{2}(\Omega_T\times \Gamma^{y}).$ This concludes the proof of the first convergence in our Lemma.

 Next, we want to prove the second convergence. Note that from the structure \eqref{ionic_model} of $\mathrm{I}_{1,ion}$ and using Proposition \ref{prop_buo}, we have
\begin{equation*}
\T_{\e}^{b}\left(\mathrm{I}_{1,ion}(v_\e)\right)=\mathrm{I}_{1,ion}\left( \T_{\e}^{b}(v_\e)\right)
\end{equation*}
Due to the strong convergence of $\T_{\e}^{b}(v_\e)$ in $L^2(\Omega_T\times \Gamma^{y}),$ we can extract a subsequence, such that $\T_{\e}^{b}(v_\e)\rightarrow v$ a.e. in $\Omega_T\times \Gamma^{y}.$ Since $\mathrm{I}_{1,ion}$ is continuous, we have
\begin{equation*}
\mathrm{I}_{1,ion}\left( \T_{\e}^{b}(v_\e)\right)\rightarrow \mathrm{I}_{1,ion}(v) \text{ a.e. in } \Omega_T\times \Gamma^{y}.
\end{equation*}
Further, we use estimate \eqref{E_vr} with property \eqref{Prop4_buo} of Proposition \ref{prop_buo} to obtain 
$$ \norm{\T_{\e}^{b}\left(\mathrm{I}_{1,ion}(v_\e)\right)}_{L^{r/(r-1)}\left(\Omega_T \times \Gamma^{y}\right)}\leq\abs{Y}^{(r-1)/r}\norm{\e^{(r-1)/r}\mathrm{I}_{1,ion}(v_\e)}_{L^{r/(r-1)}(\Gamma_{\e,T})}\leq C.$$
Hence, using a classical result (see Lemma 1.3 in \cite{lions1969}):
\begin{equation*}
\mathrm{I}_{1,ion}\left( \T_{\e}^{b}(v_\e)\right)\rightharpoonup\mathrm{I}_{1,ion}(v)\text{ weakly in } L^{r/(r-1)}(\Omega_T\times \Gamma^{y}).
\end{equation*}
Moreover, we use Vitali's Theorem to obtain the strong convergence of $\mathrm{I}_{1,ion}\left( \T_{\e}^{b}(v_\e)\right)$ to $\mathrm{I}_{1,ion}(v)$ in $L^q(\Omega_T\times \Gamma^{y}), \ \forall q\in[1,r/(r-1)).$
\end{proof}
Finally, collecting all the convergence results of $J_1,\dots,J_5$ obtained above, we pass to the limit when $\e\rightarrow 0$ in the unfolded formulation \eqref{Fv_i_r} to obtain the following limiting problem:
 \begin{equation}
\begin{aligned}
&\dfrac{\abs{\Gamma^{y}}}{\abs{Y}}\iint_{\Omega_{T}}\pt_t v\Psi_i\ dxdt
\\& +\dfrac{1}{\abs{Y}}\dfrac{1}{\abs{Z}}\iiiint_{\Omega_{T}\times Y_i \times Z_{c}}\mathrm{M}_{i} \left[ \nabla u_{i}+\nabla_y \widehat{u}_i+\nabla_z \widetilde{u}_i\right] \left[ \nabla_x \Psi_i+\Psi_1\nabla_y\Phi_1+\Psi_2 \Phi_2\nabla_z\Theta\right]  \ dxdydzdt
\\&+\dfrac{\abs{\Gamma^{y}}}{\abs{Y}}\iint_{\Omega_{T}}\mathrm{I}_{1,ion}(v)\Psi_i\ dxdt+\dfrac{\abs{\Gamma^{y}}}{\abs{Y}}\iint_{\Omega_{T}}\mathrm{I}_{2,ion}\left( w\right) \Psi_i\ dxdt
\\&=\dfrac{\abs{\Gamma^{y}}}{\abs{Y}}\iint_{\Omega_{T}}\I_{app}\Psi_i\ dxdt
\end{aligned}
\label{Fvi_psi_phi_theta}
\end{equation}
Similarly, we can prove also that the limit of \eqref{Fv_d_r} as $\e$ tends to zero, is given by:
\begin{equation}
\dfrac{\abs{\Gamma^{y}}}{\abs{Y}}\iint_{\Omega_{T}} \pt_t w \phi \ dxdt-\dfrac{\abs{\Gamma^{y}}}{\abs{Y}}\iint_{\Omega_{T}}  H(v, w) \phi \ dxdt =0.
\label{Fv_d_unf}
\end{equation}
\subsection{Extracellular problem}\label{unfextra}
The authors in \cite{bendunf19} have applied and developed the two-scale unfolding method established by Cioranescu et al. \cite{doinaunf06} on a problem defined at two scales  to obtain the homogenized model (see also \cite{doinaunf08,doinaunf12}). Whereas for the intracellular domain, we develop a three-scale approach applied to the intracellular problem to handle with the two structural levels of this domain (see Section \ref{unfintra}).  We recall the following initial extracellular problem:
\begin{equation}
 \begin{aligned}
\A_\e u_{e}^\e &=0 &\ \text{ in } \Omega_{e,T}^{\e}, \\ \mathrm{M}_{e}^{\e}\nabla u_{e}^{\e} \cdot n_e=\e\left( \pt_{t} v_\e+\I_{ion}(v_\e,{w}_\e)-\I_{app,\e}\right) &=\I_m &\ \text{on} \ \Gamma_{\e,T}, 
 \end{aligned}
 \label{pbiniextra}
 \end{equation}
 with $\A_\e=-\nabla \cdot\left( \mathrm{M}_{e}^{\e}\nabla\right),$ where the extracellular conductivity matrices $\mathrm{M}_e^{\e}=(\mathrm{m}^{pq}_{e})_{1 \leq p,q \leq d}$ defined by: $$\mathrm{M}^{\e}_{e}(x)=\mathrm{M}_{e}\left(\dfrac{x}{\e}\right), \
  a.e. \ \text{on} \ \R^d,$$
satisfying the elliptic and periodic conditions \eqref{A_M_ie}.\\

In our approach, we investigate the same technique used in \cite{bendunf19} for problem \eqref{pbiniextra}. So, we unfold the weak formulation \eqref{Fv_i_ini} of the extracellular problem using only the unfolding operators $\T_\e^e$ and $\T^b_\e$ to obtain:
  \begin{equation}
\begin{aligned}
& \dfrac{1}{\abs{Y}}\iiint_{\Omega_{T}\times \Gamma^{y}}\T_{\e}^{b}(\pt_t v_\e)\T_{\e}^{b}(\varphi_e)\ dxd\sigma_ydt
\\& +\dfrac{1}{\abs{Y}}\iint_{\Omega_{T}\times Y_e}\T_{\e}^{e}\left(\mathrm{M}_{e}^{\e}\right) \T_{\e}^{e}\left( \nabla u_{e}^{\e}\right) \T_{\e}^{e}\left(\nabla\varphi_e\right)  \ dxdydt
\\&+\dfrac{1}{\abs{Y}}\iiint_{\Omega_{T}\times \Gamma^{y}}\mathrm{I}_{1,ion}\left( \T_{\e}^{b}(v_\e)\right) \T_{\e}^{b}(\varphi_e)\ dxd\sigma_ydt
\\&+\dfrac{1}{\abs{Y}}\iiint_{\Omega_{T}\times \Gamma^{y}}\mathrm{I}_{2,ion}\left( \T_{\e}^{b}(w_\e)\right) \T_{\e}^{b}(\varphi_e)\ dxd\sigma_ydt
\\&=\dfrac{1}{\abs{Y}}\iiint_{\Omega_{T}\times \Gamma^{y}}\T_{\e}^{b}(\I_{app,\e})\T_{\e}^{b}(\varphi_e)\ dxd\sigma_ydt+R'_5-R'_4-R'_3-R'_2-R'_1
\end{aligned}
\label{Fv_e_r}
\end{equation} 
with $R'_1,\dots, R'_5 $ are similarly defined as $R_1,\dots, R_5 $ in the previous section.

 Proceeding similarly for the extracellular problem by taking into account that the test functions have the following form: \begin{equation}
\varphi_{e}^{\e}=\Psi_e(t,x)+\e\Psi_1(t,x)\Phi_1^{\e}(x),
\end{equation}
with function $\Phi_1^{\e}$  defined by: 
$$\Phi_1^{\e}(x)=\Phi_1\left( \dfrac{x}{\e}\right),$$
where $\Psi_e, \Psi_1$  are in $D(\Omega_T)$ and $\Phi_1$ in $H_{\#}^1(Y_e).$ Then, we can prove that the limit of \eqref{Fv_e_r}, as $\e$ tends zero, is given by:
 \begin{equation}
\begin{aligned}
&\dfrac{\abs{\Gamma^{y}}}{\abs{Y}}\iint_{\Omega_{T}}\pt_t v\Psi_e\ dxdt +\dfrac{1}{\abs{Y}}\iiint_{\Omega_{T}\times Y_e}\mathrm{M}_{e} \left[ \nabla u_{e}+\nabla_y \widehat{u}_e\right] \left[ \nabla \Psi_{e}+\Psi_1\nabla_y \Phi_{1}\right]  \ dxdydt
\\&+\dfrac{\abs{\Gamma^{y}}}{\abs{Y}}\iint_{\Omega_{T}}\mathrm{I}_{2,ion}\left( w\right) \Psi_e\ dxdt+\dfrac{\abs{\Gamma^{y}}}{\abs{Y}}\iint_{\Omega_{T}}\mathrm{I}_{1,ion}(v)\Psi_e\ dxdt
\\&=\dfrac{\abs{\Gamma^{y}}}{\abs{Y}}\iint_{\Omega_{T}}\I_{app}\Psi_e\ dxdt.
\end{aligned}
\label{Fve_psi_phi}
\end{equation}

\subsection{Derivation of the macroscopic bidomain model}\label{macro}
The convergence results of the previous section allow us to pass to the limit in the microscopic equations \eqref{Fv_i_ini}-\eqref{Fv_d_ini}  and to obtain the homogenized model formulated in Theorem \ref{thm_macro}.

 We first derive the macroscopic (homogenized) equation for the intracellular problem. To this end, we will find the expression of $\widehat{u}_i$ and $\widetilde{u}_i$ in terms of the homogenized solution $u_i.$ Then, we derive the cell problem from the homogenized equation \eqref{Fvi_psi_phi_theta}. Finally, we obtain the weak formulation of the corresponding macroscopic equation.
 
   We first take $\Psi_i$ equal to zero, to get:
\begin{equation}
\dfrac{1}{\abs{Y}}\dfrac{1}{\abs{Z}}\iiiint_{\Omega_{T}\times Y_i \times Z_{c}}\mathrm{M}_{i} \left[ \nabla u_{i}+\nabla_y \widehat{u}_i+\nabla_z \widetilde{u}_i\right] \left[\Psi_1\nabla_y\Phi_1 +\Psi_2 \Phi_2\nabla_z\Theta\right]  \ dxdydzdt=0.
\label{Fv_phi1_theta2}
\end{equation}
Next, to determine the explicit form of $\widetilde{u}_i$ so we take $\Psi_1$ equal to zero. Since $u_i$ and $\widehat{u}_i$ are independent of the microscopic variable $z,$ then the formulation \eqref{Fv_phi1_theta2} corresponds to the following microscopic problem:
\begin{equation}
 \begin{cases}
-\nabla_z\cdot\left(\mathrm{M}_i \nabla_z \widetilde{u}_i\right) =\overset{d}{\underset{p,q=1}{\sum}}\dfrac{\pt \mathrm{m}^{pq}_{i}}{\pt z_p}\left(\dfrac{\pt \widehat{u}_i}{\pt y_q}+\dfrac{\pt u_{i}}{\pt x_q}\right) \ \text{in} \ Z_{c},
\\ \left(\mathrm{M}_{i}\nabla_z \widetilde{u}_{i}+\mathrm{M}_{i}\nabla_y \widehat{u}_{i}+\mathrm{M}_{i}\nabla_x u_{i} \right)\cdot n_{z}=0 \ \text{on} \ \Gamma^{z}, 
\\ \widetilde{u}_i \ z\text{-periodic}. 
 \end{cases}
 \label{Azztildeu_i}
 \end{equation}
Hence, by the $z$-periodicity of $\mathrm{M}_i$ and the comptability condition, it is not difficult to establish the existence of a unique periodic solution up to an additive constant of the problem \eqref{Azztildeu_i} (see for instance the work of \cite{BaderDev}).\\
 Thus, the linearity of terms in the right of the equation \eqref{Azztildeu_i} suggests to look for $\widetilde{u}_i$ under the following form in terms of $u_i$ and $\widehat{u}_i$ :
\begin{equation} 
\widetilde{u}_i(t,x,y,z)=\theta_{i}(z)\cdot \left(\nabla_y \widehat{u}_i+\nabla_x u_{i}\right)+\tilde{u}_{0,i}(t,x,y),
\label{tildeu_i}
\end{equation}
where $\tilde{u}_{0,i}$ is a constant with respect to $z$ and each element $\theta_{i}^q$ of $\theta_{i}$ satisfies the $\delta$-cell problem:
\begin{equation}
\begin{cases}
-\nabla_z\cdot\left(\mathrm{M}_i\nabla_z \theta_{i}^q\right)  =\overset{d}{\underset{p=1}{\sum}}\dfrac{\pt \mathrm{m}^{pq}_{i}}{\pt z_p}(y,z) \ \text{in} \ Z_{c},\\ \theta_{i}^q \ y\text{- and }z\text{-periodic}, \\ \mathrm{M}_{i}\nabla_z \theta_{i}^q \cdot n_{z}=-(\mathrm{M}_{i}e_q)\cdot n_{z} \ \text{on} \ \Gamma^{z},  
 \end{cases}
 \label{Azztheta_i}
 \end{equation}
 for $q=1,\dots,d.$ Moreover, the existence and uniqueness of solution $\theta_{i}^q \in H_{\#}^1(Z_{c})$ to problem \eqref{Azztheta_i} are automatically satisfied with $H_{\#}^1(Z_{c})$ is given by \eqref{W}. 
 
Furthermore, we take $\Psi_2$ equal to zero to find the form of $\widehat{u}_i$ (note that $\psi_1$ is now chosen different from zero).  So, we replace $\widetilde{u}_i$ by its form \eqref{tildeu_i} on the formulation \eqref{Fv_phi1_theta2}. Then, we obtain a mesoscopic problem defined on the unit cell portion $Y_i$ and satisfied by $\widehat{u}_i$ as follows:
\begin{equation}
 \begin{cases}
\ -\nabla_y\cdot\left(\widetilde{\mathbf{M}}_i \nabla_y \widehat{u}_i\right)=\overset{d}{\underset{p,k=1}{\sum}} \dfrac{\pt \widetilde{\mathbf{m}}^{pk}_{i} }{\pt y_p} \dfrac{\pt u_{i}}{\pt x_k} \ \text{in} \ Y_{i},
\\ \\  \ \left(\widetilde{\mathbf{M}}_i \nabla_y \widehat{u}_{i} + \widetilde{\mathbf{M}}_i \nabla_x u_{i}\right)\cdot n_{i}=0 \ \text{ on } \ \Gamma^{y},   
 \end{cases}
 \label{Byyhatui_1}
 \end{equation}
where the coefficients of the \textbf{first-level} homogenized conductivity matrix $\widetilde{\mathbf{M}}_i=(\widetilde{\mathbf{m}}^{pk}_i)_{1\leq p,k \leq d}$ defined by: 
\begin{equation}
\widetilde{\mathbf{m}}^{pk}_i(y)=\dfrac{1}{\abs{Z}}\overset{d}{\underset{q=1}{\sum}}\displaystyle \int_{Z_{c}} \left(\mathrm{m}^{pk}_{i}+ \mathrm{m}^{pq}_{i}\dfrac{\pt \theta_{i}^{k}}{\pt z_q}\right) \ dz.
\label{Mt_i}
\end{equation}

\begin{rem}Note that the $y$-periodicity of $\widetilde{\mathbf{M}}_i$ comes from the fact that the coefficients of conductivity matrix $\mathrm{M}_{i}$ and of the function $\theta_{i}$ are $y$-periodic. Following \cite{ben,doina}, it is easy to verify that the homogenized conductivity tensors of the intracellular $\widetilde{\mathbf{M}}_{i}$ and extracellular $\widetilde{\mathbf{M}}_{e}$ spaces are symmetric and positive definite.
\end{rem}

Thus, we prove the existence and uniqueness by using same arguments from Lax-Milgram theorem (see \cite{BaderDev} for more details). \\
Hence, the linearity of terms in the right of the equation \eqref{Byyhatui_1} suggests to look for $\widehat{u}_{i}$ under the following form in terms of $u_i$:
 \begin{equation} 
\widehat{u}_{i}(t,x,y)=\chi_{i}(y)\cdot \nabla_x u_{i}+\widehat{u}_{0,i}(t,x),
\label{hatui_1}
\end{equation}
where $\widehat{u}_{0,i}$ a constant with respect to $y$ and each element $\chi_{i}^k$ of $\chi_{i}$ satisfies the following $\e$-cell problem:
\begin{equation}
\begin{cases}
 -\nabla_y\cdot\left(\widetilde{\mathbf{M}}_i \nabla_y \chi_{i}^k\right)  =\overset{d}{\underset{p=1}{\sum}}\dfrac{\pt \widetilde{\mathbf{m}}_{i}^{pk}}{\pt y_p} \ \text{in} \ Y_{i},\\ \\ \widetilde{\mathbf{M}}_{i} \nabla_y \chi_{i}^k \cdot n_{i}=-\left(\widetilde{\mathbf{M}}_{i} e_k \right) \cdot n_{i} \ \text{on} \ \Gamma^{y},
 \end{cases}
 \label{Byychi_i}
 \end{equation}
 for $e_k,k=1,\dots,d,$ the standard canonical basis in $\R^d.$ Since the matrix $\widetilde{\mathbf{M}}_{i}$ is positive definite, so we can prove the existence and uniqueness of the solution $\chi_{i}^k \in H_{\#}^1(Y_{i})$ to problem \eqref{Byychi_i}.

\begin{rem} At this point, we deduce that this method is used to homogenize the problem with respect to $z$ and then with respect to $y$. We remark also that  allows to obtain the effective properties at $\delta$-structural level and which become the input values in order to find the effective behavior of the cardiac tissue.
\end{rem}

 Finally, inserting the form \eqref{tildeu_i}-\eqref{hatui_1} of $\widetilde{u}_i$ and $\widehat{u}_i$ into \eqref{Fvi_psi_phi_theta} and setting $\Psi_1, \Psi_2$ equals to zero, one obtains the weak formulation of the homogenized equation for the intracellular problem:
 \begin{equation}
\begin{aligned}
&\mu_{m}\iint_{\Omega_{T}}\pt_t v\Psi_i\ dxdt+\iint_{\Omega_{T}}\doublewidetilde{\mathbf{M}}_{i} \nabla u_{i} \cdot \nabla \Psi_i  \ dxdt+\mu_{m}\iint_{\Omega_{T}}\mathrm{I}_{1,ion}\left( v\right) \Psi_i\ dxdt
\\&+\mu_{m}\iint_{\Omega_{T}}\mathrm{I}_{2,ion}(w)\Psi_i\ dxdt=\mu_{m}\iint_{\Omega_{T}}\I_{app}\Psi_i\ dxdt
\end{aligned}
\label{Fvhomi}
\end{equation}
with $\mu_{m}=\abs{\Gamma^{y}}/\abs{Y}$ and the coefficients of the \textbf{second-level} homogenized conductivity matrix $\doublewidetilde{\mathbf{M}}_i=\left( \doublewidetilde{\mathbf{m}}^{pq}_i\right)_{1\leq p,q \leq d}$ defined by:
\begin{equation}
\begin{aligned}
\doublewidetilde{\mathbf{m}}^{pq}_i & :=\dfrac{1}{\abs{Y}}\overset{d}{\underset{k=1}{\sum}}\displaystyle \int_{Y_{i}}  \left(\widetilde{\mathbf{m}}_i^{pk} \dfrac{\pt \chi_{i}^q}{\pt y_k}(y)+\widetilde{\mathbf{m}}_i^{pq}\right) dy
\\ & = \dfrac{1}{\abs{Y}} \dfrac{1}{\abs{Z}}\overset{d}{\underset{k,\ell=1}{\sum}}\displaystyle \int_{Y_{i}} \int_{Z_{c}}   \left[\displaystyle \left(\mathrm{m}^{pk}_{i}+ \mathrm{m}^{p\ell}_{i}\dfrac{\pt \theta_{i}^{k}}{\pt z_\ell}\right) \dfrac{\pt \chi_{i}^q}{\pt y_k}(y)+\left(\mathrm{m}^{pq}_{i}+ \mathrm{m}^{p\ell}_{i}\dfrac{\pt \theta_{i}^{q}}{\pt z_\ell}\right)\right] \ dzdy
\end{aligned}
\label{Mb_i}
\end{equation} 
with the coefficients of the conductivity matrix $\widetilde{\mathbf{M}}_i=\left(\widetilde{\mathbf{m}}^{pk}_i\right)_{1\leq p,k \leq d}$ defined by \eqref{Mt_i}.

Similarly, we obtain the second homogenized equation for the extracellular problem:\begin{equation}
\begin{aligned}
&\mu_{m}\iint_{\Omega_{T}}\pt_t v\Psi_e\ dxdt+\iint_{\Omega_{T}}\widetilde{\mathbf{M}}_{e} \nabla u_{e} \cdot \nabla \Psi_e  \ dxdt+\mu_{m}\iint_{\Omega_{T}}\mathrm{I}_{2,ion}\left( w\right) \Psi_e\ dxdt
\\&+\mu_{m}\iint_{\Omega_{T}}\mathrm{I}_{1,ion}(v)\Psi_e\ dxdt=\mu_{m}\iint_{\Omega_{T}}\I_{app}\Psi_e\ dxdt
\end{aligned}
\label{Fvhome}
\end{equation}
with $\mu_{m}=\abs{\Gamma^{y}}/\abs{Y}$ and the coefficients of the homogenized conductivity matrices $\widetilde{\mathbf{M}}_e=\left( \widetilde{\mathbf{m}}^{pk}_e\right)_{1\leq p,k \leq d}$ defined by:
\begin{equation}
\widetilde{\mathbf{m}}^{pk}_e:=\dfrac{1}{\abs{Y}}\overset{d}{\underset{q=1}{\sum}}\displaystyle\int_{Y_e}\left( \mathrm{m}_e^{pk}+\mathrm{m}^{pq}_{e}\dfrac{\pt \chi_e^k}{\pt y_q}\right) \ dy.
\label{Mb_e}
\end{equation}
each element $\chi_{e}^k \in H_{\#}^1(Y_e)$ of $\chi_{e}$ satisfies the following $\e$-cell problem:
\begin{equation}
\begin{cases}
 -\nabla_y\cdot\left(\mathbf{M}_e \nabla_y \chi_{e}^k\right)  =\overset{d}{\underset{p=1}{\sum}}\dfrac{\pt \mathbf{m}_{e}^{pk}}{\pt y_p} \ \text{in} \ Y_{e},\\ \\ \mathbf{M}_{e} \nabla_y \chi_{e}^k \cdot n_{e}=-\left(\mathbf{M}_{e} e_k \right) \cdot n_{e} \ \text{on} \ \Gamma^{y},
 \end{cases}
 \label{Bxxchi_e}
 \end{equation}
 for $e_k,k=1,\dots,d,$ the standard canonical basis in $\R^d.$

\section{Conclusion}
Many biological and physical phenomena arise in highly heterogeneous media, the properties of which vary on three (or more) length scales. 
In this paper, an important homogenization technique have been established for predicting the bioelectrical behaviors of the cardiac tissue with multiple small-scale configurations. Furthermore, we have presented via the unfolding homogenization a rigorous mathematical justification  for the results obtained in a recent work \cite{BaderDev} based a three-scale asymptotic homogenization method. These main mathematical models describe the bioelectrical activity of the heart, from the microscopic activity of ion channels of the cellular membrane to the macroscopic properties in the whole heart. We have described how reaction-diffusion systems can be derived from microscopic models of cellular aggregates by unfolding homogenization method on three different scales. 
  
  The present study has some limitations and is open to several improvements. For example, analytical formulas have been found for an ideal particular geometry at the mesoscale and microscale. Nevertheless, the natural next step is to consider more realistic geometries by solving the appropriate cellular problems analytically and numerically.
 
  
 As future plans, we intend to address the resolution of a novel problem, called "tridomain model", by taking into account the presence of gap junctions as connection between adjacent cardiac cells. We want to investigate existence and uniqueness of solutions of the tridomain equations, including commonly used ionic model, namely the FitzHugh-Nagumo model. An additional step could be the derivation, using the homogenization theory, of the macroscopic behaviors of heart tissue.


%
%

\section*{Acknowledgments}
We would like to thank the anonymous referee for his careful reading and constructive comments.

\bibliographystyle{plain}
 \bibliography{Hom}

\appendix

\section{Compactness result for the space $L^p(\Omega,B)$}\label{appB}
In this part, we give a characterization of relatively compact sets $F$ in $L^p(\Omega, B)$ for $p\in[1;+\infty),$ $\Omega\subset \R^d$ open and bounded set and $B$ a Banach space. 
\begin{prop}[Kolmogorov-Riesz type compactness result] Let $\Omega \subset \R^d$ be an open and bounded set. Let $F \subset L^{p}(\Omega,B)$ for a Banach space B and $p\in[1;+\infty).$ For $f\in F$ and $h\in\R^d,$ we define $\tau_{h}f(x):=f(x+h).$ Then $F$ is relatively compact in $L^{p}(\Omega,B)$ if and only if
\begin{itemize}
\item[$(i)$] for every measurable set $ C\subset \Omega$ the set $\lbrace \int_C f dx \ 
: \ f \in F \rbrace$ is relatively compact in $B,$
\item[$(ii)$] for all $\lambda>0,$ $h\in \R^d$ and $h_i\geq 0,$ $i=1,\dots,d,$ there holds
$$\underset{f\in F}{\sup} \norm{\tau_{h}f-f}_{L^p\left(\Omega_\lambda^h,B\right)}\rightarrow 0, \text{ for } h\rightarrow 0, $$
where $\Omega_\lambda^h:=\lbrace x\in \Omega_\lambda : x+h \in \Omega_\lambda\rbrace$ and $\Omega_\lambda:=\lbrace x\in \Omega : dist(x,\pt \Omega)>\lambda\rbrace,$
\item[$(iii)$] for $\lambda>0,$ there holds $\underset{f\in F}{\sup} \int_{\Omega\setminus\Omega_\lambda} \abs{f(x)}^p dx\rightarrow 0$ for $\lambda \rightarrow 0.$
\end{itemize}
\label{kolmo}
\end{prop}
\begin{proof}
The proof of the proposition can be found as Corollary 2.5 in \cite{maria}.
\end{proof}

\end{document}